\documentclass[11pt]{article}
\usepackage{setspace}
\usepackage{amsmath} 
\usepackage{amssymb}
\usepackage{latexsym}
\usepackage{xcolor}
\usepackage{fancyhdr}
\allowdisplaybreaks 

 \usepackage{comment}

\def\XXint#1#2#3{{\setbox0=\hbox{$#1{#2#3}{\int}$} 
\vcenter{\hbox{$#2#3$}}\kern-.5\wd0}}   

 \numberwithin{equation}{section}
\newtheorem{theorem}[equation]{Theorem}
\newtheorem{proposition}[equation]{Proposition}
\newtheorem{definition}[equation]{Definition}
\newtheorem{remark}[equation]{Remark}
\newtheorem{lemma}[equation]{Lemma}
\newtheorem{corollary}[equation]{Corollary}

\title{A Helmholtz-Weyl decomposition for H\"{o}lder spaces. A nonvariational approach.
 } 
 
\author{  
Massimo Lanza de Cristoforis
\\
Dipartimento di Matematica `Tullio Levi-Civita', 
\\
Universit\`a degli Studi di Padova, 
\\
Via Trieste 63, Padova 35121, 
Italy. 
\\
E-mail: mldc@math.unipd.it   }

\date{\ }

\begin{document}
 \maketitle


\noindent
{\bf Abstract:}  We prove a form of the fundamental theorem of vector calculus
 for H\"older continuous vector fields  with no assumption  on the derivatives
 that exploits Schauder spaces with negative exponents. Then by developing an idea of von Wahl in the context of Lebesgue spaces, we prove a form of the Helmholtz Weyl decomposition Theorem
  for spaces of H\"{o}lder continuous functions, with explicit formulas for the projection operators.

\vspace{\baselineskip}



\section{Introduction}
 The Helmholtz-Weyl decomposition of a three dimensional vector field is a classical topic and we refer to
book of Galdi \cite[\S III.1]{Ga11} for a historical account. See also
Alberti, Brown, Marletta and Wood \cite[\S 3]{AlBrMaWo19}. Here we plan to prove 
a Helmholtz-Weyl decomposition   
by writing an explicit form of the projection operators in the frame of H\"{o}lder spaces.\par
  
The first step consists in proving the fundamental theorem of vector calculus
 for H\"older continuous vector fields in the form of Theorem \ref{thm:cudirepl}. 
 The main point here is that   we make no assumption whatsoever on the derivatives of the involved 
   H\"older continuous vector fields.\par
 
 Then we are ready to prove our form of the Helmholtz Weyl decomposition Theorem \ref{thm:hewedeco0a}. We do so by proving a formula for the projections 
  that develops from an idea of von Wahl \cite{vo90} in the context of Lebesgue spaces, although here we have to modify it for H\"{o}lder spaces are not separable.

  The paper is organized as follows. Section \ref{sec:prelnot} is a section of preliminaries and notation. In Section \ref{sec:dndbdry}, we introduce a distributional form of the normal derivative 
  for H\"{o}lder continuous  functions   of \cite{La25}.  In section \ref{sec:acsilah-1a}, we introduce some preliminaries on the acoustic layer potentials. In Section \ref{sec:tandeho}, we introduce the tangential derivatives of H\"{o}lder continuous functions on the boundary. In section \ref{sec:acvolpot} we introduce some   known properties of the acoustic volume potential.
  In section \ref{sec:futvean}, we prove  the fundamental Theorem \ref{thm:cudirepl} of vector calculus   for H\"{o}lder continuous  functions. In Section \ref{sec:prelNeum} we present a technical result on the Neumann problem that we need in the sequel. In Section \ref{sec:helweyl}, we prove the Helmholtz-Weyl decomposition 
  Theorem \ref{thm:hewedeco0a} for  H\"older spaces.

\section{Preliminaries and notation}\label{sec:prelnot} Unless otherwise specified,  we assume  throughout the paper that
\[
n\in {\mathbb{N}}\setminus\{0,1\}\,,
\]
where ${\mathbb{N}}$ denotes the set of natural numbers including $0$. 
If $X$, $Y$ and $Z$ are normed spaces, then ${\mathcal{L}}(X,Y)$ denotes the space of linear and continuous maps from $X$ to $Y$ and ${\mathcal{L}}^{(2)}(X\times Y, Z)$ 
denotes the space of bilinear and continuous maps from $X\times Y$ to $Z$ with their usual operator norm (cf.~\textit{e.g.}, \cite[pp.~16, 621]{DaLaMu21}). $I$ denotes the identity map. ${\mathrm{Im}}$  and ${\mathrm{Ker}}$ denote  the image (or range) and the kernel (or null space) of an operator, respectively.
  $|A|$ denotes the operator norm of a matrix $A$ with real (or complex) entries,  $A^{t}$ denotes the transpose matrix of $A$.	As customary in the literature of electromagnetism, we  set
\begin{eqnarray}\label{eq:scvepr}
 \lefteqn{a\cdot b\equiv \sum_{j=1}^na_jb_j\qquad\forall a\equiv (a_1,\dots,a_n), b\equiv (b_1,\dots,b_n)\in      {\mathbb{C}}^n\,,
 }
 \\ \nonumber
 \lefteqn{a\bar{\wedge} b\equiv 
 \left(a_2b_3-a_3b_2,a_3b_1-a_1b_3,a_1b_2-a_2b_1\right)
 }
 \\ \nonumber
&&\qquad\qquad\qquad\qquad\qquad\qquad 
  \qquad\forall a\equiv (a_1,a_2,a_3), b\equiv (b_1,b_2,b_3)\in      {\mathbb{C}}^3\,,
\end{eqnarray}
thus not only for real components.  For all $r\in]0,+\infty[$, $ x\in{\mathbb{R}}^{n}$, 
$x_{j}$ denotes the $j$-th coordinate of $x$, 
$| x|$ denotes the Euclidean modulus of $ x$ in
${\mathbb{R}}^{n}$, and ${\mathbb{B}}_{n}( x,r)$ denotes the ball $\{
y\in{\mathbb{R}}^{n}:\, | x- y|<R\}$. 
 
 Let $\Omega$ be an open subset of ${\mathbb{R}}^n$. $C^{1}(\Omega)$ denotes the set of continuously differentiable functions from $\Omega$ to ${\mathbb{C}}$. 
 If $s\in {\mathbb{N}}\setminus\{0\}$, $f\in \left(C^{1}(\Omega)\right)^{s} $, then   $Df$ denotes the Jacobian matrix of $f$.   
 Instead, we denote the gradient of a function  by  $\nabla$. 
For us the gradient is always a column vector. 
Let  $\eta\equiv
(\eta_{1},\dots ,\eta_{n})$ be in ${\mathbb{N}}^{n}$, $|\eta |\equiv
\eta_{1}+\dots +\eta_{n}  $. Then $D^{\eta} f$ denotes
$\frac{\partial^{|\beta|}f}{\partial
x_{1}^{\eta_{1}}\dots\partial x_{n}^{\eta_{n}}}$, and we also use standard abbreviations as $\partial_{x_l}\equiv \frac{\partial}{\partial x_l}$.\par

For the (classical) definition of   open Lipschitz subset of ${\mathbb{R}}^n$ and of   open subset of ${\mathbb{R}}^n$ 
   of class $C^{m}$ or of class $C^{m,\alpha}$
  and of the H\"{o}lder and Schauder spaces $C^{m,\alpha}(\overline{\Omega})$
  on the closure $\overline{\Omega}$ of  an open set $\Omega$ and 
  of the H\"{o}lder and Schauder spaces
   $C^{m,\alpha}(\partial\Omega)$ 
on the boundary $\partial\Omega$ of an open set $\Omega$ for some $m\in{\mathbb{N}}$, $\alpha\in ]0,1]$, we refer for example to
    Dalla Riva, the author and Musolino  \cite[\S 2.3, \S 2.6, \S 2.7, \S 2.9, \S 2.11, \S 2.13,   \S 2.20]{DaLaMu21}.  If $m\in {\mathbb{N}}$, 
 $C^{m}_b(\overline{\Omega})$ denotes the space of $m$-times continuously differentiable functions from $\Omega$ to ${\mathbb{C}}$ such that all 
the partial derivatives up to order $m$ have a bounded continuous extension to    $\overline{\Omega}$ and we set
\[
\|f\|_{   C^{m}_{b}(
\overline{\Omega} )   }\equiv
\sum_{|\eta|\leq m}\, \sup_{x\in \overline{\Omega}}|D^{\eta}f(x)|
\qquad\forall f\in C^{m}_{b}(
\overline{\Omega} )\,.
\]
If $\alpha\in ]0,1]$, then 
$C^{m,\alpha}_b(\overline{\Omega})$ denotes the space of functions of $C^{m}_{b}(
\overline{\Omega}) $  such that the  partial derivatives of order $m$ are $\alpha$-H\"{o}lder continuous in $\Omega$. Then we equip $C^{m,\alpha}_{b}(\overline{\Omega})$ with the norm
\[
\|f\|_{  C^{m,\alpha}_{b}(\overline{\Omega})  }\equiv 
\|f\|_{  C^{m }_{b}(\overline{\Omega})  }
+\sum_{|\eta|=m}|D^{\eta}f|_{\alpha}\qquad\forall f\in C^{m,\alpha}_{b}(\overline{\Omega})\,,
\]
where $|D^{\eta}f|_{\alpha}$ denotes the $\alpha$-H\"{o}lder constant of the partial derivative $D^{\eta}f$ of order $\eta$ of $f$ in $\Omega$. If $\Omega$ is bounded, we obviously have $C^{m }_{b}(\overline{\Omega})=C^{m } (\overline{\Omega})$ and $C^{m,\alpha}_{b}(\overline{\Omega})=C^{m,\alpha} (\overline{\Omega})$. 
  Then $C^{m,\alpha}_{{\mathrm{loc}}}(\overline{\Omega }) $  denotes 
the space  of those functions $f\in C^{m}(\overline{\Omega} ) $ such that $f_{|\overline{\Omega'}} $ belongs to $
C^{m,\alpha}(   \overline{ \Omega' }   )$ for all bounded open subsets $\Omega'$ of ${\mathbb{R}}^n$ such that $\overline{\Omega'}\subseteq\overline{\Omega}$. 
The space of complex valued functions of class $C^m$ with compact support in an open set $\Omega$ of ${\mathbb{R}}^n$ is denoted $C^m_c(\Omega)$ and similarly for $C^\infty_c(\Omega)$. We also set ${\mathcal{D}}(\Omega)\equiv C^\infty_c(\Omega)$. Then the dual ${\mathcal{D}}'(\Omega)$ is known to be the space of distributions in $\Omega$. The support either of a function or of a distribution is denoted by the abbreviation `${\mathrm{supp}}$'.  We also set
\begin{equation}\label{eq:exto}
\Omega^-\equiv {\mathbb{R}}^n\setminus\overline{\Omega}\,,
\end{equation} 
for the exterior of $\Omega$. If $\Omega$ is a bounded open Lipschitz subset of ${\mathbb{R}}^n$, then $\Omega$  is known to have has at most a finite number of pairwise disjoint  connected components, which we denote by $\Omega_{1}$,\dots, $\Omega_{\varkappa^{+}}$ and which are open and $\Omega^-$ is known to have  at most a finite number of  pairwise disjoint  connected components, which we denote by $(\Omega^{-})_{0}$,
$(\Omega^{-})_{1}$, \dots, $(\Omega^{-})_{\varkappa^{-}}$ and which are open.   One and only one of such connected components is unbounded. We denote it by $(\Omega^{-})_{0}$  (cf.~\textit{e.g.}, \cite[Lemma~2.38]{DaLaMu21}).\par  

We denote by $\nu_\Omega$ or simply by $\nu$ the outward unit normal of $\Omega$ on $\partial\Omega$. Then $\nu_{\Omega^-}=-\nu_\Omega$ is the outward unit normal of $\Omega^-$ on $\partial\Omega=\partial\Omega^-$.

Morever, we retain the standard notation for the Lebesgue spaces $L^p$ for $p\in [1,+\infty]$ (cf.~\textit{e.g.}, Folland \cite[Chapter~6]{Fo99}, \cite[\S 2.1]{DaLaMu21}) and
$m_n$ denotes the
$n$ dimensional Lebesgue measure.\par

We now summarize the definition and some elementary properties of the Schau\-der space $C^{-1,\alpha}(\overline{\Omega})$ 
by following the presentation of Dalla Riva, the author and Musolino \cite[\S 2.22]{DaLaMu21}.
\begin{definition} 
\label{defn:sch-1}\index{Schauder space!with negative exponent}
 Let $\alpha\in]0,1]$. Let $\Omega$ be a bounded open subset of ${\mathbb{R}}^{n}$. We denote by $C^{-1,\alpha}(\overline{\Omega})$ the subspace 
 \[
 \left\{
 f_{0}+\sum_{j=1}^{n}\frac{\partial}{\partial x_{j}}f_{j}:\,f_{j}\in 
 C^{0,\alpha}(\overline{\Omega})\ \forall j\in\{0,\dots,n\}
 \right\}\,,
 \]
 of the space of distributions ${\mathcal{D}}'(\Omega)$  in $\Omega$ and we set
 \begin{eqnarray}
\label{defn:sch-2}
\lefteqn{
\|f\|_{  C^{-1,\alpha}(\overline{\Omega})  }
\equiv\inf\biggl\{\biggr.
\sum_{j=0}^{n}\|f_{j}\|_{ C^{0,\alpha}(\overline{\Omega})  }
:\,
}
\\ \nonumber
&&\qquad\qquad\qquad\quad
f=f_{0}+\sum_{j=1}^{n}\frac{\partial}{\partial x_{j}}f_{j}\,,\ 
f_{j}\in C^{0,\alpha}(\overline{\Omega})\ \forall j\in \{0,\dots,n\}
\biggl.\biggr\}\,.
\end{eqnarray}
\end{definition}
$(C^{-1,\alpha}(\overline{\Omega}), \|\cdot\|_{  C^{-1,\alpha}(\overline{\Omega})  })$ is known to be a Banach space and  is continuously embedded into ${\mathcal{D}}'(\Omega)$. Also, the definition of the norm $\|\cdot\|_{  C^{-1,\alpha}(\overline{\Omega})  }$ implies that $C^{0,\alpha}(\overline{\Omega})$ is continuously embedded into $C^{-1,\alpha}(\overline{\Omega})$ and that the partial differentiation $\frac{\partial}{\partial x_{j}}$ is continuous from 
$C^{0,\alpha}(\overline{\Omega})$ to $C^{-1,\alpha}(\overline{\Omega})$ for all $j\in\{1,\dots,n\}$.  Generically, the  elements of $C^{-1,\alpha}(\overline{\Omega})$ for $\alpha\in]0,1[$ are not integrable functions, but distributions in $\Omega$.  Then we have the following statement of \cite[Prop.~3.1]{La24c} that shows that the elements of $C^{-1,\alpha}(\overline{\Omega}) $  can be extended to elements of the dual of $C^{1,\alpha}(\overline{\Omega})$.  Here we emphasize that   the elements of $C^{-1,\alpha}(\overline{\Omega}) $ belong   to the dual of ${\mathcal{D}}(\Omega)$ by definition.
\begin{proposition}\label{prop:nschext}
 Let $\alpha\in]0,1[$. Let $\Omega$ be a bounded open Lipschitz subset of ${\mathbb{R}}^{n}$.  There exists one and only one  linear and continuous extension operator $E^\sharp_\Omega$ from $C^{-1,\alpha}(\overline{\Omega})$ to $\left(C^{1,\alpha}(\overline{\Omega})\right)'$ such that
 \begin{eqnarray}\label{prop:nschext2}
\lefteqn{
\langle E^\sharp_\Omega[f],v\rangle 
}
\\ \nonumber
&&\ \
 =
\int_{\Omega}f_{0}v\,dx+\int_{\partial\Omega}\sum_{j=1}^{n} (\nu_{\Omega})_{j}f_{j}v\,d\sigma
 -\sum_{j=1}^{n}\int_{\Omega}f_{j}\frac{\partial v}{\partial x_j}\,dx
\quad \forall v\in C^{1,\alpha}(\overline{\Omega}) 
\end{eqnarray}
for all $f=  f_{0}+\sum_{j=1}^{n}\frac{\partial}{\partial x_{j}}f_{j}\in C^{-1,\alpha}(\overline{\Omega}) $. Moreover, 
\begin{equation}\label{prop:nschext1}
E^\sharp_\Omega[f]_{|\Omega}=f\,, \ i.e.,\ 
\langle E^\sharp_\Omega[f],v\rangle =\langle f,v\rangle \qquad\forall v\in {\mathcal{D}}(\Omega)
\end{equation}
for all $f\in C^{-1,\alpha}(\overline{\Omega})$ and
\begin{equation}\label{prop:nschext3}
\langle E^\sharp_\Omega[f],v\rangle =\langle f,v\rangle =\int_\Omega fv\,dx
\qquad\forall v\in C^{1,\alpha}(\overline{\Omega})
\end{equation}
for all $f\in C^{0,\alpha}(\overline{\Omega})$.
\end{proposition}
When no ambiguity can arise, we simply write $E^\sharp$ instead of $E^\sharp_\Omega$. To see why the extension operator $E^\sharp$ can be considered as `canonical', we refer to \cite[Prop.~7]{La24d}. 
{\color{black}In case $f\equiv (f_1,\dots,f_n)\in C^{-1,\alpha}(\overline{\Omega}, {\mathbb{C}}^n)$, we set
\[
E^\sharp[f]\equiv (E^\sharp[f_1],\dots,E^\sharp[f_n])\,.
\]}
Then we have the following immediate corollary.
\begin{corollary}\label{corol:gengr}
 Let $\alpha\in]0,1[$. Let $\Omega$ be a bounded open Lipschitz subset of ${\mathbb{R}}^{n}$. If $g\in C^{0,\alpha}(\overline{\Omega})$ and $l\in\{1,\dots,n\}$, then
 \begin{equation}\label{corol:gengr1}
\langle E^\sharp_\Omega[\frac{\partial g}{\partial x_l}],v\rangle=\int_{\partial\Omega}(\nu_\Omega)_l g v\,d\sigma
-
\int_\Omega g\frac{\partial v}{\partial x_l}\,dx
\quad \forall v\in C^{1,\alpha}(\overline{\Omega}) \,.
\end{equation}
\end{corollary}
Then we can easily prove  the identity of the following statement.
\begin{proposition}\label{prop:gendidis}
 Let $\alpha\in]0,1[$. Let $\Omega$ be a bounded open Lipschitz subset of ${\mathbb{R}}^{n}$.
 If $A\in C^{0,\alpha}(\overline{\Omega},{\mathbb{C}}^n)$ and $v\in 
C^{1,\alpha}(\overline{\Omega})$, then
\[
\langle E^\sharp_\Omega[{\mathrm{div}}\, A],v\rangle 
+\int_\Omega A\cdot Dv\,dx=\int_{\partial\Omega}\nu_\Omega\cdot A\, v\,d\sigma 
\]
(cf.~(\ref{eq:scvepr})).
\end{proposition}
{\bf Proof.} By the definition of ${\mathrm{div}}$ and by formula (\ref{corol:gengr1}), we have
\[
\langle E^\sharp_\Omega[{\mathrm{div}}\, A],v\rangle
=\sum_{l=1}^n \langle E^\sharp_\Omega[\frac{\partial A_l}{\partial x_l}],v\rangle
=\sum_{l=1}^n\left\{
\int_{\partial\Omega}(\nu_\Omega)_l A_l v\,d\sigma
-
\int_\Omega A_l\frac{\partial v}{\partial x_l}\,dx
\right\}\,.
\]\hfill  $\Box$ 

\vspace{\baselineskip}

We now show that we can write a (distributional) form of the Divergence Theorem  by exploiting the operator $E^\sharp_\Omega$. Once more, with no assumptions on the first order derivatives.
\begin{theorem}[of the Divergence in distributional form]\label{thm:div0a}
 Let $\alpha\in]0,1[$. Let $\Omega$ be a bounded open Lipschitz subset of ${\mathbb{R}}^{n}$. If $A$ belongs to $ C^{0,\alpha}(\overline{\Omega}, {\mathbb{C}}^n)$, then
 \begin{equation}
\langle E^\sharp_\Omega[{\mathrm{div}}\, A],1\rangle =\int_{\partial\Omega}A\cdot\nu_\Omega\,d\sigma
\end{equation}
(cf.~(\ref{eq:scvepr})).
\end{theorem}
{\bf Proof.} It suffices to apply Proposition \ref{prop:gendidis} with $v=1$.\hfill  $\Box$ 

\vspace{\baselineskip}

We now show that we can write a (distributional) form of   other classical identities for $\alpha$-H\"{o}lder continuous vector fields by exploiting the operator $E^\sharp_\Omega$. Once more, with no assumptions on the first order derivatives.
\begin{theorem}\label{thm:veiniddis}
Let $\alpha\in]0,1[$. Let $\Omega$ be a bounded open Lipschitz subset of ${\mathbb{R}}^{3}$.  Then the following statements hold (cf.~(\ref{eq:scvepr})).
\begin{enumerate}
\item[(i)] If $A \in C^{0,\alpha}(\overline{\Omega},{\mathbb{C}}^3)$, then
\[
\langle E^\sharp_\Omega[{\mathrm{curl}}\,A], 1\rangle =\int_{\partial\Omega}\nu_\Omega\bar{\wedge} A\,d\sigma\,.
\]
\item[(ii)]  If $A\in C^{0,\alpha}(\overline{\Omega},{\mathbb{C}}^3)$ and
$v\in C^{1,\alpha}(\overline{\Omega},{\mathbb{C}}^3)$, then
\[
\langle E^\sharp_\Omega[{\mathrm{curl}}\,A], v\rangle
-
\int_\Omega A\cdot  {\mathrm{curl}}\,  v\,dx=\int_{\partial\Omega}(\nu_\Omega\bar{\wedge}A)\cdot v\,d\sigma\,.
\]
\item[(iii)]  If $A\in C^{0,\alpha}(\overline{\Omega},{\mathbb{C}}^3)$ and
$v\in C^{1,\alpha}(\overline{\Omega},{\mathbb{C}})$, then
\[
\langle E^\sharp_\Omega[{\mathrm{curl}}\,A], v\rangle
-
\int_\Omega A\bar{\wedge} \nabla v\,dx=\int_{\partial\Omega}(\nu_\Omega\bar{\wedge}A)  v\,d\sigma\,.
\]
\end{enumerate}
\end{theorem}
{\bf Proof.} (ii) By the definition of ${\mathrm{curl}}$ and by formula (\ref{corol:gengr1}), we have
\begin{eqnarray*} 
\lefteqn{
\langle E^\sharp_\Omega[{\mathrm{curl}}\,A], v\rangle
=\langle E^\sharp_\Omega[\frac{\partial A_3}{\partial x_2}
-\frac{\partial A_2}{\partial x_3}
], v_1\rangle
}
\\ \nonumber
&&\qquad 
-
\langle E^\sharp_\Omega[\frac{\partial A_3}{\partial x_1}
-\frac{\partial A_1}{\partial x_3}
], v_2\rangle
+
\langle E^\sharp_\Omega[\frac{\partial A_2}{\partial x_1}
-\frac{\partial A_1}{\partial x_2}
], v_3\rangle
\\ \nonumber
&&\qquad 
=\int_{\partial\Omega}(\nu_\Omega)_2A_3v_1\,d\sigma-\int_\Omega A_3\frac{\partial v_1}{\partial x_2}\,dx
\\ \nonumber
&&\qquad\quad
-\left(
\int_{\partial\Omega}(\nu_\Omega)_3A_2v_1\,d\sigma-\int_\Omega A_2\frac{\partial v_1}{\partial x_3}\,dx
\right)
\\ \nonumber
&&\qquad\quad
-\biggl\{\int_{\partial\Omega}(\nu_\Omega)_1A_3v_2\,d\sigma-\int_\Omega A_3\frac{\partial v_2}{\partial x_1}\,dx
\\ \nonumber
&&\qquad\quad
-\left(
\int_{\partial\Omega}(\nu_\Omega)_3A_1v_2\,d\sigma-\int_\Omega A_1\frac{\partial v_2}{\partial x_3}\,dx
\right)\biggr\}
\\ \nonumber
&&\qquad\quad 
+\int_{\partial\Omega}(\nu_\Omega)_1A_2v_3\,d\sigma-\int_\Omega A_2\frac{\partial v_3}{\partial x_1}\,dx
\\ \nonumber
&&\qquad\quad
-\left(
\int_{\partial\Omega}(\nu_\Omega)_2A_1v_3\,d\sigma-\int_\Omega A_1\frac{\partial v_3}{\partial x_2}\,dx
\right)
\\ \nonumber
&&\qquad 
=\int_{\partial\Omega}\left((\nu_\Omega)_2A_3-(\nu_\Omega)_3A_2\right)v_1\,d\sigma
-\int_\Omega A_3\frac{\partial v_1}{\partial x_2}-A_2\frac{\partial v_1}{\partial x_3}\,dx
\\ \nonumber
&&\qquad \quad
-\biggl\{
\int_{\partial\Omega}\left((\nu_\Omega)_1A_3-(\nu_\Omega)_3A_1\right)v_2\,d\sigma
-\int_\Omega A_3\frac{\partial v_2}{\partial x_1}-A_1\frac{\partial v_2}{\partial x_3}\,dx
\biggr\}
\\ \nonumber
&&\qquad \quad
+\int_{\partial\Omega}\left((\nu_\Omega)_1A_2-(\nu_\Omega)_2A_1\right)v_3\,d\sigma
-\int_\Omega A_2\frac{\partial v_3}{\partial x_1}-A_1\frac{\partial v_3}{\partial x_2}\,dx\,.
\\ \nonumber
&&\qquad
=\int_{\partial\Omega}(\nu_\Omega \bar{\wedge} A)\cdot v\,d\sigma
+
\int_\Omega \left(\frac{\partial v_3}{\partial x_2}-\frac{\partial v_2}{\partial x_3}\right)A_1\,dx
\\ \nonumber
&&\qquad\quad
-
\int_\Omega \left(\frac{\partial v_3}{\partial x_1}-\frac{\partial v_1}{\partial x_3}\right)A_2\,dx
+
\int_\Omega \left(\frac{\partial v_2}{\partial x_1}-\frac{\partial v_1}{\partial x_2}\right)A_3\,dx
\\ \nonumber
&&\qquad
=\int_{\partial\Omega}(\nu_\Omega \bar{\wedge} A)\cdot v\,d\sigma
+
\int_\Omega A\cdot {\mathrm{curl}}\,v\,dx\,.
\end{eqnarray*}
   The equality in (iii) can be obtained by that in (ii) by taking the vector field $v$ of (ii) equal to the vector fields $(v,0,0)$, $(0,v,0)$ and $(0,0,v)$, where $v$ is the scalar function of (iii). The equality in (i) can be obtained by that in (iii) by setting $v=1$. \hfill  $\Box$ 

\vspace{\baselineskip}

{\color{black}
Next we introduce the following subspace of those $\alpha$-H\"{o}lder continuous functions for which we can define the normal derivative on the boundary (cf.~\cite[Defn.~5.6]{La24c}).

\begin{definition}\label{defn:c0ade}
 Let   $\alpha\in ]0,1]$. Let $\Omega$ be a bounded open  subset of ${\mathbb{R}}^{n}$. Let
 \begin{eqnarray}\label{defn:c0ade1}
C^{0,\alpha}(\overline{\Omega})_\Delta
&\equiv&\biggl\{u\in C^{0,\alpha}(\overline{\Omega}):\,\Delta u\in C^{-1,\alpha}(\overline{\Omega})\biggr\}\,,
\\ \nonumber
\|u\|_{ C^{0,\alpha}(\overline{\Omega})_\Delta }
&\equiv& \|u\|_{ C^{0,\alpha}(\overline{\Omega})}
+\|\Delta u\|_{C^{-1,\alpha}(\overline{\Omega})}
\qquad\forall u\in C^{0,\alpha}(\overline{\Omega})_\Delta\,.
\end{eqnarray}
\end{definition}
Since $C^{0,\alpha}(\overline{\Omega})$ and $C^{-1,\alpha}(\overline{\Omega})$ are Banach spaces,   $\left(\|u\|_{ C^{0,\alpha}(\overline{\Omega})_\Delta }, \|\cdot \|_{ C^{0,\alpha}(\overline{\Omega})_\Delta }\right)$ is a Banach space. }
Next we introduce the following {\color{black}two} approximation lemmas.
\begin{lemma}\label{lem:apr1a}
 Let $\alpha\in]0,1]$. Let $m\in {\mathbb{N}}\setminus\{0\}$, $h\in\{0,\dots,m\}$.
 Let $\Omega$ be a bounded open subset of ${\mathbb{R}}^n$ of class $C^{m,\alpha}$.  If $g\in C^{h,\alpha}(\overline{\Omega})$, then there exists a sequence $\{g_j\}_{j\in {\mathbb{N}}}$ in 
 $C^{\infty}(\overline{\Omega})$ such that
 \begin{equation}\label{lem:apr1a1}
 \sup_{j\in {\mathbb{N}}}\|g_j\|_{
 C^{h,\alpha}(\overline{\Omega}) 
 }<+\infty\,,\qquad
 \lim_{j\to\infty}g_j=g\quad\text{in}\ C^{h,\beta}(\overline{\Omega})\quad \forall\beta\in]0,\alpha[\,.
 \end{equation}
\end{lemma}
For a proof, we refer for example to \cite[Lem.~A.3]{La24c}.  Then we have the following certainly known approximation lemma 
for functions in Schauder spaces on the boundary.
\begin{lemma}\label{lem:aprbdma}
 Let $\alpha\in]0,1]$, $m\in {\mathbb{N}}$. 
 Let $\Omega$ be a bounded open subset of ${\mathbb{R}}^n$ of class $C^{\max\{1,m\},\alpha}$. If $f\in C^{m,\alpha}(\partial\Omega)$, then there exists a sequence $\{f_j\}_{j\in {\mathbb{N}}}$ in 
 $C^{\max\{1,m\},\alpha}(\partial\Omega)$ such that
 \begin{equation}\label{lem:aprbdma1}
 \sup_{j\in {\mathbb{N}}}\|f_j\|_{
 C^{m,\alpha}(\partial\Omega) 
 }<+\infty\,,\qquad
 \lim_{j\to\infty}f_j=f\quad\text{in}\ C^{m,\beta}(\partial\Omega)\quad \forall\beta\in]0,\alpha[\,.
 \end{equation}
\end{lemma}
{\bf Proof.} For case $m=0$, we refer for example to \cite[Lem.~A.25]{La24b}. Now let $m\geq 1$. 
Since $\Omega$ is of class $C^{m,\alpha}$ and $f\in C^{m,\alpha}(\partial\Omega)$, there exists
$\tilde{f}\in C^{m,\alpha}(\overline{\Omega})$ such that $\tilde{f}_{|\partial\Omega}=f$ (cf. \textit{e.g.}, \cite[2.72]{DaLaMu21}). Then Lemma \ref{lem:apr1a} implies the existence of a  sequence $\{\tilde{f}_j\}_{j\in {\mathbb{N}}}$ in 
 $C^{\infty}(\overline{\Omega})\subseteq C^{m,\alpha}(\overline{\Omega})$ such that
\begin{equation}\label{lem:aprbdma2}
 \sup_{j\in {\mathbb{N}}}\|\tilde{f}_j\|_{
 C^{m,\alpha}(\overline{\Omega}) 
 }<+\infty\,,\qquad
 \lim_{j\to\infty}\tilde{f}_j=\tilde{f}\quad\text{in}\ C^{m,\beta}(\overline{\Omega})\quad \forall\beta\in]0,\alpha[\,.
 \end{equation}
Then the continuity of the restriction operator
from $C^{m,\gamma}(\overline{\Omega})$ to $C^{m,\gamma}(\partial\Omega)$ for $\gamma=\alpha$ and $\gamma=\beta$ imply the validity of (\ref{lem:aprbdma1}) with $f_j\equiv \tilde{f}_{j|\partial\Omega}$ for all $j\in {\mathbb{N}}$.\hfill  $\Box$ 

\vspace{\baselineskip}

Next we introduce the  following Lemma that is  well known and is an immediate consequence of the H\"{o}lder inequality.
\begin{lemma}\label{lem:caincl}
 Let $m\in {\mathbb{N}}$, $\alpha\in ]0,1[$. Let $\Omega$ be a bounded open  subset of ${\mathbb{R}}^{n}$ of class $C^{\max\{m,1\},\alpha}$.  Then the canonical inclusion  ${\mathcal{J}}$ from the Lebesgue space $L^1(\partial\Omega)$ of integrable functions in $\partial\Omega$ to $(C^{m,\alpha}(\partial\Omega))'$ that takes $\mu$ to the functional ${\mathcal{J}}[\mu]$ defined by 
 \begin{equation}\label{lem:caincl1}
\langle {\mathcal{J}}[\mu],v\rangle \equiv \int_{\partial\Omega}\mu v\,d\sigma\qquad\forall v\in C^{
 m,
\alpha}(\partial\Omega)\,,
\end{equation}
is linear continuous and injective.
\end{lemma}
As customary, we say  that ${\mathcal{J}}[\mu]$ is the `distribution' that is canonically associated to $\mu$ and we omit the indication of the inclusion map ${\mathcal{J}}$. By Lemma \ref{lem:caincl}, the space $C^{0,\alpha}(\partial\Omega)$ is continuously embedded into $(C^{m,\alpha}(\partial\Omega))'$.

\section{Distributional normal derivatives on the boundary}\label{sec:dndbdry}

Next we plan to introduce the normal derivative of the functions in  $C^{0,\alpha}(\overline{\Omega})_\Delta$ as in \cite{La24c}. To do so,    we introduce the (classical) interior Steklov-Poincar\'{e} operator (also known as interior  Dirichlet-to-Neumann map).
\begin{definition}\label{defn:cinspo}
 Let $\alpha\in]0,1[$.  Let  $\Omega$ be a  bounded open subset of ${\mathbb{R}}^{n}$ of class $C^{1,\alpha}$. The classical interior Steklov-Poincar\'{e} operator is defined to be the operator $S_{\Omega,+}$ from
 \begin{equation}\label{defn:cinspo1}
C^{1,\alpha}(\partial\Omega)\quad\text{to}\quad C^{0,\alpha}(\partial\Omega)
\end{equation}
that takes $v\in C^{1,\alpha}(\partial\Omega)$ to the function 
 \begin{equation}\label{defn:cinspo2}
S_{\Omega,+}[v](x)\equiv \frac{\partial  }{\partial\nu}{\mathcal{G}}_{\Omega,d,+}[v](x)\qquad\forall x\in\partial\Omega\,,
\end{equation}
where ${\mathcal{G}}_{d,+}[v]$ is the only solution $v^\sharp\in C^{1,\alpha}(\overline{\Omega})$  of the Dirichlet problem
\[
\left\{
\begin{array}{ll}
 \Delta v^\sharp=0 & \text{in}\ \Omega\,,
 \\
v^\sharp_{|\partial\Omega} =v& \text{on}\ \partial\Omega \,.
\end{array}
\right.
\]

 \end{definition}
   Since   
    ${\mathcal{G}}_{\Omega,d,+}$ is linear and continuous from $C^{1,\alpha}(\partial\Omega)$ to 
   $C^{1,\alpha}(\overline{\Omega})$ (cf.~\textit{e.g.}, \cite[Thm.~4.8]{La24b}) and the classical normal derivative is continuous from $C^{1,\alpha}(\overline{\Omega})$ to $C^{0,\alpha}(\partial\Omega)$, then $S_{\Omega,+}[\cdot]$ is linear and continuous from 
  $C^{1,\alpha}(\partial\Omega)$ to $C^{0,\alpha}(\partial\Omega)$. Then we have the following definition of \cite[(41)]{La24c}.
  \begin{definition}\label{defn:conoderdedu}
 Let $\alpha\in]0,1[$.  Let  $\Omega$ be a  bounded open subset of ${\mathbb{R}}^{n}$ of class $C^{1,\alpha}$. If  $u\in C^{0}(\overline{\Omega})$ and $\Delta u\in  C^{-1,\alpha}(\overline{\Omega})$, then we define the distributional  normal derivative $\partial_{\nu_\Omega} u$
 of $u$ to be the only element of the dual $(C^{1,\alpha}(\partial\Omega))'$ that satisfies the following equality
 \begin{equation}\label{defn:conoderdedu1}
\langle \partial_{\nu_\Omega} u ,v\rangle \equiv\int_{\partial\Omega}uS_{\Omega,+}[v]\,d\sigma
+\langle E^\sharp_\Omega[\Delta u],{\mathcal{G}}_{\Omega,d,+}[v]\rangle 
\qquad\forall v\in C^{1,\alpha}(\partial\Omega)\,.
\end{equation}
\end{definition}
 The normal derivative of Definition \ref{defn:conoderdedu} extends the classical one in the sense that if $u\in C^{1,\alpha}(\overline{\Omega})$, then under the assumptions on $\alpha$ and $\Omega$  of Definition \ref{defn:conoderdedu}, we have 
 \begin{equation}\label{lem:conoderdeducl1}
 \langle \partial_{\nu_\Omega} u ,v\rangle =\int_{\partial\Omega}\frac{\partial u}{\partial\nu_\Omega}v\,d\sigma
  \quad\forall v\in C^{1,\alpha}(\partial\Omega)\,,
\end{equation}
where $\frac{\partial u}{\partial\nu_\Omega}$ in the right hand side denotes the classical normal derivative of $u$ on $\partial\Omega$ (cf.~\cite[Lem.~5.5]{La24c}). In the sequel, we use the classical symbol $\frac{\partial u}{\partial\nu_\Omega}$ also for  $\partial_{\nu_\Omega} u$ when no ambiguity can arise.\par  
  
Next we introduce the  function space $V^{-1,\alpha}(\partial\Omega)$ on the boundary of $\Omega$ for the normal derivatives of the functions of $C^{0,\alpha}(\overline{\Omega})_\Delta$ as in  \cite[Defn.~13.2, 15.10, Thm.~18.1]{La24b}.
\begin{definition}\label{defn:v-1a}
Let   $\alpha\in ]0,1[$. Let $\Omega$ be a bounded open  subset of ${\mathbb{R}}^{n}$ of class $C^{1,\alpha}$. Let 
\begin{eqnarray}\label{defn:v-1a1}
 \lefteqn{V^{-1,\alpha}(\partial\Omega)\equiv \biggl\{\mu_0+S_{\Omega,+}^t[\mu_1]:\,\mu_0, \mu_1\in C^{0,\alpha}(\partial\Omega)
\biggr\}\,,
}
\\ \nonumber
\lefteqn{
\|\tau\|_{  V^{-1,\alpha}(\partial\Omega) }
\equiv\inf\biggl\{\biggr.
 \|\mu_0\|_{ C^{0,\alpha}(\partial\Omega)  }+\|\mu_1\|_{ C^{0,\alpha}(\partial\Omega)  }
:\,
 \tau=\mu_0+S_{\Omega,+}^t[\mu_1]\biggl.\biggr\}\,,
 }
 \\ \nonumber
 &&\qquad\qquad\qquad\qquad\qquad\qquad\qquad\qquad\qquad
 \forall \tau\in  V^{-1,\alpha}(\partial\Omega)\,,
\end{eqnarray}
where $S_{\Omega,+}^t$ is the transpose map of $S_{\Omega,+}$.
\end{definition}
As shown in \cite[\S 13]{La24b},  $(V^{-1,\alpha}(\partial\Omega), \|\cdot\|_{  V^{-1,\alpha}(\partial\Omega)  })$ is a Banach space. By definition of the norm, $C^{0,\alpha}(\partial\Omega)$ is continuously embedded into $V^{-1,\alpha}(\partial\Omega)$. Moreover, we have the following statement  of \cite[Prop.~6.6]{La24c}  on the continuity of the normal derivative on $C^{0,\alpha}(\overline{\Omega})_\Delta$. 
\begin{proposition}\label{prop:ricodnu}
 Let   $\alpha\in ]0,1[$. Let $\Omega$ be a bounded open  subset of 
 ${\mathbb{R}}^{n}$ of class $C^{1,\alpha}$. Then the distributional normal derivative   $\partial_{\nu_\Omega}$ is a continuous surjection of $C^{0,\alpha}(\overline{\Omega})_\Delta$ onto $V^{-1,\alpha}(\partial\Omega)$ and there exists a linear and continuous map $Z$ from $V^{-1,\alpha}(\partial\Omega)$ to $C^{0,\alpha}(\overline{\Omega})_\Delta $ such that
 \begin{equation}\label{prop:ricodnu1}
\partial_{\nu_\Omega} Z[g]=g\qquad\forall g\in V^{-1,\alpha}(\partial\Omega)\,,
\end{equation}
\textit{i.e.}, $Z$ is a right inverse of   $\partial_{\nu_\Omega}$.   
\end{proposition}
In order to introduce a condition that is equivalent to that of Definition \ref{defn:conoderdedu}, we   introduce the following  sub-space of  $C^{1,\alpha}(\overline{\Omega})$.  
\begin{definition}\label{defn:c1ade}
 Let   $\alpha\in ]0,1[$. Let $\Omega$ be a bounded open  subset of ${\mathbb{R}}^{n}$ of class $C^{1,\alpha}$. Let
 \begin{eqnarray}\label{defn:c1ade1}
C^{1,\alpha}(\overline{\Omega})_\Delta
&\equiv&\biggl\{u\in C^{1,\alpha}(\overline{\Omega}):\,\Delta u\in C^{0,\alpha}(\overline{\Omega})\biggr\}\,,
\\ \nonumber
\|u\|_{ C^{1,\alpha}(\overline{\Omega})_\Delta }
&\equiv& \|u\|_{ C^{1,\alpha}(\overline{\Omega})}
+\|\Delta u\|_{C^{0,\alpha}(\overline{\Omega})}
\qquad\forall u\in C^{1,\alpha}(\overline{\Omega})_\Delta\,.
\end{eqnarray}
\end{definition}
Then    $\left(\|u\|_{ C^{1,\alpha}(\overline{\Omega})_\Delta }, \|\cdot \|_{ C^{1,\alpha}(\overline{\Omega})_\Delta }\right)$ is a Banach space (cf.~\cite[\S 5.1]{La24c}). Then we have the following statement of   \cite[Prop.~5.15]{La24c}.
\begin{proposition}\label{prop:node1adeq}
 Let   $\alpha\in ]0,1[$. Let $\Omega$ be a bounded open  subset of 
 ${\mathbb{R}}^{n}$ of class $C^{1,\alpha}$. Let $E_\Omega$ be a linear map from $C^{1,\alpha}(\partial\Omega)$ to $ C^{1,\alpha}(\overline{\Omega})_\Delta$ such that
 \begin{equation}\label{prop:node1adeq0}
 E_\Omega[f]_{|\partial\Omega}=f\qquad\forall f\in C^{1,\alpha}(\partial\Omega)\,.
 \end{equation}
 If $u\in C^{0,\alpha}(\overline{\Omega})_\Delta$, then the distributional  normal derivative $\partial_{\nu_\Omega} u$
 of $u$  is characterized by the validity of the following equality
 \begin{eqnarray}\label{prop:node1adeq1}
 \lefteqn{
\langle \partial_{\nu_\Omega} u ,v\rangle =\int_{\partial\Omega}u
\frac{\partial}{\partial\nu_\Omega}E_\Omega[v]
\,d\sigma
}
\\ \nonumber
&&\qquad\qquad
+\langle E^\sharp_\Omega[\Delta u],E_\Omega[v]\rangle -\int_\Omega\Delta (E_\Omega[v]) u\,dx
\qquad\forall v\in C^{1,\alpha}(\partial\Omega)\,.
\end{eqnarray}
\end{proposition}
   We note that in equality (\ref{prop:node1adeq1}) we can use the `extension' operator $E_\Omega$ that we prefer and that accordingly equality (\ref{prop:node1adeq1}) is independent of the specific choice of $E_\Omega$. When we deal with problems for the Laplace operator, a good choice is
 $E_\Omega={\mathcal{G}}_{\Omega,d,+}$ (cf.~\textit{e.g.}, \cite[Thm.~4.8]{La24b}), so that the last term in the right hand side of (\ref{prop:node1adeq1}) disappears.  
 In particular, under the assumptions of Proposition \ref{prop:node1adeq}, an extension operator as $E_\Omega$ always exists.
 For the definition of   the normal derivative on the boundary the functions of $C^{0,\alpha}_{
{\mathrm{loc}}	}(\overline{\Omega^-})_\Delta$ with $\alpha\in]0,1[$ in case $\Omega$ is a bounded open subset of ${\mathbb{R}}^n$ of class $C^{1,\alpha}$, we refer to  \cite[\S 3]{La25}.

\section{Preliminaries on the  acoustic layer potentials}\label{sec:acsilah-1a}

Let  $\alpha\in]0,1]$. Let  $\Omega$ be a bounded open subset of ${\mathbb{R}}^{n}$ of class $C^{1,\alpha}$. Let  $r_{|\partial\Omega}$  be the restriction map  from ${\mathcal{D}}({\mathbb{R}}^n)$ to $C^{1,\alpha}(\partial\Omega)$. Let $\lambda\in {\mathbb{C}}$. If $S_{n,\lambda} $ is a fundamental solution 
of the operator $\Delta+\lambda$ and $\mu\in (C^{1,\alpha}(\partial\Omega))'$, then the  (distributional) single layer  potential relative to $S_{n,\lambda} $ and $\mu$ is the distribution
\[
v_\Omega[S_{n,\lambda} ,\mu]=(r_{|\partial\Omega}^t\mu)\ast S_{n,\lambda}  \in {\mathcal{D}}'({\mathbb{R}}^n) 
\]
 and we  set
\begin{eqnarray}\label{eq:dsila}
v_\Omega^+[S_{n,\lambda} ,\mu]&\equiv&\left((r_{|\partial\Omega}^t\mu)\ast S_{n,\lambda}  \right)_{|\Omega}
\qquad\text{in}\ \Omega\,,
\\ \nonumber
v_\Omega^-[S_{n,\lambda} ,\mu] &\equiv&
\left((r_{|\partial\Omega}^t\mu)\ast S_{n,\lambda}  \right)_{|\Omega^-}
\qquad\text{in}\ \Omega^-\,.
\end{eqnarray}
It is also known that the restriction of $v_\Omega[S_{n,\lambda} ,\mu]$ to ${\mathbb{R}}^n\setminus\partial\Omega$ equals the (distribution that is associated to) the function
\[
 \langle (r_{|\partial\Omega}^t\mu)(y),S_{n,\lambda} (\cdot-y)\rangle \,.
\]
In the cases in which both $v_\Omega^+[S_{n,\lambda} ,\mu]$ and $v_\Omega^-[S_{n,\lambda} ,\mu]$ admit a continuous extension to $\overline{\Omega}$ and to $\overline{\Omega^-}$, respectively, we still use the symbols $v_\Omega^+[S_{n,\lambda} ,\mu]$ and $v_\Omega^-[S_{n,\lambda} ,\mu]$ for the continuous extensions and if the values of $v_\Omega^\pm[S_{n,\lambda} ,\mu](x)$ coincide for each $x\in\partial\Omega$, then we set
\[
V_\Omega[S_{n,\lambda} ,\mu](x)\equiv v_\Omega^+[S_{n,\lambda} ,\mu](x)=v_\Omega^-[S_{n,\lambda} ,\mu]^-(x)\qquad\forall x\in\partial\Omega\,.
\]
If $\mu$ is continuous, then it is known that $v_\Omega[S_{n,\lambda} ,\mu ]$ is continuous in ${\mathbb{R}}^n$ (cf.~\textit{e.g.}, \cite[Lem.~4.2 (i), Lem.~6.2]{DoLa17}).   For the classical  properties of the acoustic single layer potential, we refer for example to  the paper \cite{DoLa17} of the author and Dondi and to \cite[\S 5]{La25a}. If $\lambda=0$, \textit{i.e.} if $\Delta+\lambda $ is the Laplace operator, then we usually  consider the fundamental solution $S_n$ of $\Delta$ that is delivered by the formula
 \[
S_{n}(\xi)\equiv
\left\{
\begin{array}{lll}
\frac{1}{s_{n}}\ln  |\xi| \qquad &   \forall \xi\in 
{\mathbb{R}}^{n}\setminus\{0\},\quad & {\mathrm{if}}\ n=2\,,
\\
\frac{1}{(2-n)s_{n}}|\xi|^{2-n}\qquad &   \forall \xi\in 
{\mathbb{R}}^{n}\setminus\{0\},\quad & {\mathrm{if}}\ n>2\,,
\end{array}
\right.
\]
where $s_{n}$ denotes the $(n-1)$ dimensional measure of 
$\partial{\mathbb{B}}_{n}(0,1)$.  Next we introduce the following (classical) technical statement on the  double layer potential (cf.~\textit{e.g.}, \cite[Thm.~5.1]{La25a}).
\begin{theorem}\label{thm:dlay}
 Let  $\alpha\in]0,1[$. Let $\Omega$ be a bounded open subset of  ${\mathbb{R}}^n$ of class $C^{1,\alpha}$. Let $\lambda\in {\mathbb{C}}$.  Let $S_{n,\lambda} $ be a fundamental solution 
of the operator $\Delta+\lambda$. If $\mu\in C^{0,\alpha}(\partial\Omega)$, and
\begin{equation}\label{thm:dlay1}
w_\Omega[S_{n,\lambda} ,\mu](x)\equiv\int_{\partial\Omega}\frac{\partial}{\partial\nu_{\Omega,y}}\left(S_{n,\lambda} (x-y)\right)\mu(y)\,d\sigma_y\qquad\forall x\in {\mathbb{R}}^n\,,
\end{equation}
where
\[
\frac{\partial}{\partial \nu_{\Omega,y} }
\left(S_{n,\lambda}(x-y)\right)\equiv
  -DS_{n,\lambda}(x-y) \nu_{\Omega}(y) \qquad\forall (x,y)\in\mathbb{R}^n\times\partial\Omega\,, x\neq y\,,
 \]
then the restriction 
$w_\Omega[S_{n,\lambda} ,\mu]_{|\Omega}$ can be extended uniquely to a function  
$w^{+}_\Omega [S_{n,\lambda} ,\mu]$ of class $ C^{0,\alpha}(\overline{\Omega})$ and $w_\Omega[S_{n,\lambda} ,\mu]_{|{\mathbb{R}}^n\setminus\overline{\Omega}}$ can be extended uniquely to a function  
$w^{-}_\Omega [S_{n,\lambda} ,\mu]$ of class  $
C^{0,\alpha}_{ {\mathrm{loc}} }(\overline{\Omega^{-}})$. Moreover, we have the following jump relation
\begin{equation}\label{thm:dlay1a}
w^{\pm}_\Omega [S_{n,\lambda},\mu](x)
=\pm\frac{1}{2}\mu(x)+w_\Omega[S_{n,\lambda},\mu](x)
\qquad\forall x\in\partial\Omega\,.
\end{equation}
 \end{theorem}
For the classical  properties of the acoustic double layer potential, we refer for example to  the paper \cite[\S 7]{DoLa17} of the author and Dondi and to \cite[\S 5]{La25a}. We also set
\begin{equation}\label{thm:dlaybdry}
W_\Omega[S_{n,\lambda} ,\mu](x)\equiv w_\Omega[S_{n,\lambda} ,\mu](x)\qquad\forall x\in\partial\Omega\,.
\end{equation}
In order to shorten our notation, we introduce the following abbreviation. Let $m\in {\mathbb{N}}$,   $\alpha\in]0,1[$. Let $\Omega$ be a bounded open subset of ${\mathbb{R}}^{n}$ of class $C^{\max\{1,m\},\alpha}$. Then we set
\begin{equation}\label{eq:vm-1a}
V^{m-1,\alpha}(\partial\Omega)\equiv\left\{
\begin{array}{ll}
 C^{m-1,\alpha}(\partial\Omega)& \text{if}\ m\geq 1\,,
 \\
 V^{-1,\alpha}(\partial\Omega)& \text{if}\ m=0\,.
\end{array}
\right.
\end{equation}
\begin{remark}\label{rem:wtnotation}
{\em  Let  $m\in{\mathbb{N}}$, $\alpha\in]0,1[$. Let $\Omega$ be a bounded open subset of ${\mathbb{R}}^{n}$ of class $C^{\max\{1,m\},\alpha}$. Let $\lambda\in {\mathbb{C}}$. Let $S_{n,\lambda}$ be a fundamental solution of $\Delta+\lambda$. Then
 $W_\Omega[S_{n,\lambda},\cdot]$ is known to be compact in $C^{\max\{1,m\},\alpha}(\partial\Omega)$ (cf.~\textit{e.g.},   \cite[Cor.~9.1]{DoLa17}) and $W_\Omega^t[S_{n,\lambda},\cdot]$ denotes the transpose map to the operator $W_\Omega[S_{n,\lambda},\cdot]$ with respect to the natural duality pairing
\begin{equation}\label{rem:wtnotation1}
 \left(
 \left(C^{\max\{1,m\},\alpha}(\partial\Omega)\right)',C^{\max\{1,m\},\alpha}(\partial\Omega)
 \right)\,.
 \end{equation}
 Via the identification map ${\mathcal{J}}$ of Lemma \ref{lem:caincl}, $V^{m-1,\alpha}(\partial\Omega)$
 can be regarded as a subspace of $\left(C^{\max\{1,m\},\alpha}(\partial\Omega)\right)'$ and 
 it is   known that $W_\Omega^t[S_{n,\lambda},\cdot]$ is compact in $V^{m-1,\alpha}(\partial\Omega)$ (cf.~\textit{e.g.}, \cite[Cor.~10.1]{DoLa17} in case $m\geq 1$  and \cite[Cor.~8.5]{La25a} in case $m=0$). 
 Then    $W_\Omega^t[S_{n,\lambda},\cdot]_{|V^{m-1,\alpha}(\partial\Omega)}$ denotes the transpose map to the operator $W_\Omega[S_{n,\lambda},\cdot]$ with respect to the  duality pairing
\begin{equation}\label{rem:wtnotation2}
 \left(
 V^{m-1,\alpha}(\partial\Omega),C^{\max\{1,m\},\alpha}(\partial\Omega)
 \right)\,.
\end{equation}
(cf.~\textit{e.g.}, Kress~\cite[Defn.~4.5]{Kr14} for the definition of transpose map with respect to a duality pairing).
}\end{remark}

\section{Tangential derivatives of a H\"{o}lder continuous function}\label{sec:tandeho}

We now follow the scheme of Dalla Riva, the author and Musolino \cite[\S 2.22]{DaLaMu21} in order to introduce   the tangential derivatives of a function defined on the boundary of an open  subset of  ${\mathbb{R}}^n$ of class $C^1$. 
If $j,l\in\{1,\dots,n\}$,  then $M_{jl}$ denotes the tangential derivative 
 operator from $C^{1}(\partial\Omega)$ to $C^0(\partial\Omega)$ that takes $f$ to  
 \begin{equation}
\label{mlr}
M_{jl}[f]\equiv (\nu_\Omega)_{j}\frac{\partial\tilde{f}}{\partial x_{l}}-
(\nu_\Omega)_{l}\frac{\partial\tilde{f}}{\partial x_{j}}\qquad {\text{on}}\ \partial\Omega\,,
\end{equation}
where  $\tilde{f}$ is any   extension of class $C^1$ of $f$  to an open neighborhood of $\partial\Omega$. We note that $M_{jl}[f]$ is independent of the specific choice of $\tilde{f}$ (cf.~\textit{e.g.}, Dalla Riva, the author and Musolino
 \cite[\S 2.21]{DaLaMu21}). Then the tangential gradient of $f$ is delivered by the fomula
 \begin{eqnarray}\label{defn:tangrad} 
\lefteqn{
{\mathrm{grad}}_{\partial\Omega}f(x)=\nabla\tilde{f}(x)-\left(\nabla\tilde{f}(x)\cdot\nu_\Omega(x)\right)\nu_\Omega(x)
}
\\ \nonumber
&&\qquad 
=\sum_{l=1}^n M_{lr}[f](x)(\nu_\Omega)_l(x)e_r
\qquad\forall x\in\partial\Omega\,,\  \forall f\in C^1(\partial\Omega  )	\,,
\end{eqnarray}
where  $\{e_1,\cdots,e_n\}$ is the canonical basis of ${\mathbb{R}}^n$ (cf.~Kirsch and Hettlich \cite[A.5]{KiHe15},  reference \cite[\S 2.21]{DaLaMu21} with Dalla Riva and Musolino). 
In the specific case in which $n=3$, we note that
\begin{equation}\label{defn:tangrad1} 
\nu\bar{\wedge}\nabla\tilde{f}
=\nu\bar{\wedge}{\mathrm{grad}}_{\partial\Omega}f\qquad\forall f\in C^1(\partial\Omega  )\,.
\end{equation}
In particular, both ${\mathrm{grad}}_{\partial\Omega}f$ and
$\nu\bar{\wedge}\nabla\tilde{f}$ are independent of the specific choice of $\tilde{f}$ so that  
$\nu\bar{\wedge}\nabla\tilde{f}$ is simply written as $\nu\bar{\wedge}\nabla f$ by abuse of notation.
Next we prove we   the following elementary statement that indicates the geometrical meaning of the tangential derivatives.   
\begin{lemma}\label{lem:taden=3}
Let $\Omega$ be a bounded open  subset of ${\mathbb{R}}^{3}$ of class $C^1$.  If $\phi \in C^{1}(\partial\Omega)$, then
\begin{equation}\label{lem:taden=31}
\nu\bar{\wedge}{\mathrm{grad}}_{\partial\Omega}\phi
=M_{23}[\phi]e_1
-M_{13}[\phi] e_2+M_{12}[\phi] e_3\qquad \text{on}\ \partial\Omega\,,
\end{equation}
where $\{e_1,e_2,e_3\}$ is the canonical basis of ${\mathbb{R}}^3$ (cf.~\cite[(2.84)]{DaLaMu21}). 
\end{lemma}
 We also note that if
  $U\in \{\Omega,\Omega^-\}$ and $\varphi \in C^{1}(\overline{U})$, then $\varphi$ has an extension of class $C^1({\mathbb{R}}^3)$ and 
\begin{eqnarray}\label{lem:taden=31a}
\lefteqn{
\nu_\Omega\bar{\wedge}{\color{black}\bigl( } \nabla \varphi{\color{black}\bigr)}_{|\partial\Omega}
=
\nu\bar{\wedge}{\mathrm{grad}}_{\partial\Omega}\phi_{{\color{black}|\partial\Omega}}
}
\\ \nonumber
&&\qquad\qquad
=M_{23}[\varphi_{|\partial\Omega}]e_1
-M_{13}[\varphi_{|\partial\Omega}] e_2+M_{12}[\phi_{|\partial\Omega}] e_3\qquad \text{on}\ \partial\Omega\,.
\end{eqnarray}
We also need the following  consequence of the Divergence Theorem.
\begin{lemma}
\label{lem:gagre}
Let $n\in {\mathbb{N}}\setminus\{0,1\}$.  
Let $\Omega$ be a bounded open  subset of ${\mathbb{R}}^{n}$ of class $C^1$. If $\varphi$, $\psi\in C^{1}(\partial\Omega)$, then
\[
\int_{\partial\Omega}M_{lr}[\varphi]\psi\,d\sigma=-
\int_{\partial\Omega}\varphi M_{lr}[\psi] \,d\sigma
\]
for all $l,r\in \{1,\dots,n\}$.
\end{lemma}
For a proof, we refer for example to Dalla Riva, the author and Musolino
 \cite[Lem.~2.86]{DaLaMu21}. Then we can introduce the following known definition.

\begin{definition}\label{defn:tanderdis}
  Let $\alpha\in]0,1]$. Let $\Omega$ be a bounded open   subset of ${\mathbb{R}}^{n}$ of class $C^{1,\alpha}$.  If $\tau\in (C^{1,\alpha}(\partial\Omega))'$, then we set
\begin{equation}\label{defn:tanderdis1}
\langle M_{lr}[\tau],\varphi \rangle \equiv-\langle \tau,M_{lr}[\varphi]\rangle \qquad\forall \varphi\in C^{1,\alpha}(\partial\Omega)\,.
\end{equation}
for all $l,r\in \{1,\dots,n\}$.
\end{definition}
 Then $M_{lr}$ is linear and continuous from $C^{0,\alpha}(\partial\Omega)$
  to the dual $(C^{1,\alpha}(\partial\Omega))'$ and we want to show that
  $M_{lr}$ is linear and continuous from $C^{0,\alpha}(\partial\Omega)$
  to $V^{-1,\alpha}(\partial\Omega)$. To do so, we need some preliminary statement.
\begin{proposition}\label{prop:vmlrmu}
Let $\alpha\in]0,1[$. Let $\Omega$ be a bounded open   subset of ${\mathbb{R}}^{n}$ of class $C^{1,\alpha}$. Let $\mu\in C^{0,\alpha}(\partial\Omega)$, $l,r\in \{1,\dots,n\}$. Then the following statements hold.
\begin{enumerate}
\item[(i)] 
\begin{eqnarray}\label{prop:vmlrmu1}
\lefteqn{
v_\Omega[S_n,M_{lr}[\mu]](x)=
 \langle (r_{|\partial\Omega}^tM_{lr}[\mu])(y),S_{n} (x-y)\rangle
}
\\ \nonumber
&&\   
=\frac{\partial}{\partial x_r}v_\Omega[S_n,(\nu_\Omega)_l\mu](x)
-\frac{\partial}{\partial x_l}v_\Omega[S_n,(\nu_\Omega)_r\mu](x)
\ \ \forall x\in {\mathbb{R}}^{n}\setminus\partial\Omega\,.
\end{eqnarray}
\item[(ii)] $v_\Omega^+[S_n,M_{lr}[\mu]]\in C^{0,\alpha}(\overline{\Omega})$, $v_\Omega^-[S_n,M_{lr}[\mu]]\in C^{0,\alpha}_{{\mathrm{loc}}}(\overline{\Omega^-})$ and
\begin{eqnarray*} 
\lefteqn{
v_\Omega^\pm[S_n,M_{lr}[\mu]](x)=
\frac{\partial}{\partial x_r}v_\Omega^\pm[S_n,(\nu_\Omega)_l\mu](x)
-\frac{\partial}{\partial x_l}v_\Omega^\pm[S_n,(\nu_\Omega)_r\mu](x)
}
\\ \nonumber
&&\qquad\qquad 
={\mathrm{p.v.}}\int_{\partial\Omega}
\frac{\partial}{\partial x_r}S_n(x-y)(\nu_\Omega)_l(y)\mu(y)
\\ \nonumber
&&\qquad\qquad\qquad\quad 
-\frac{\partial}{\partial x_l}S_n(x-y)(\nu_\Omega)_r(y)\mu(y)\,d\sigma_y\qquad\forall x\in\partial\Omega
 \,.
\end{eqnarray*}
\item[(iii)] 
$\langle M_{lr}[\mu],V_\Omega[S_n,\psi]\rangle=\int_{\partial\Omega}V_\Omega[S_n,M_{lr}[\mu]]\psi\,d\sigma
$
 for all $\psi\in C^{0,\alpha}(\partial\Omega)$.
\end{enumerate}
\end{proposition}
{\bf Proof.} (i) By the definition of $M_{lr}$ and by the classical differentiation theorem for integrals depending on a parameter, we have
\begin{eqnarray*} 
\lefteqn{
 \langle (r_{|\partial\Omega}^tM_{lr}[\mu])(y),S_{n} (x-y)\rangle
 =-\int_{\partial\Omega}\mu(y)M_{lr,y}S_{n} (x-y)\,d\sigma
}
\\ \nonumber
&& 
=-\int_{\partial\Omega}\mu(y)(\nu_\Omega)_l(y)\frac{\partial  }{\partial y_r}\left(S_{n} (x-y)\right)
-\mu(y)(\nu_\Omega)_r(y)\frac{\partial  }{\partial y_l}\left(S_{n} (x-y)\right)\,d\sigma
\\ \nonumber
&& 
=\int_{\partial\Omega}\mu(y)(\nu_\Omega)_l(y)\frac{\partial  }{\partial x_r} S_{n} (x-y) 
-\mu(y)(\nu_\Omega)_r(y)\frac{\partial  }{\partial x_l}S_{n} (x-y)\,d\sigma
\\ \nonumber
&& 
=\frac{\partial}{\partial x_r}v_\Omega[S_n,(\nu_\Omega)_l\mu](x)
-\frac{\partial}{\partial x_l}v_\Omega[S_n,(\nu_\Omega)_r\mu](x)
\qquad\forall x\in {\mathbb{R}}^{n}\setminus\partial\Omega\,.
\end{eqnarray*}
(ii) Since the  the components of $\nu_\Omega$ belong to $C^{0,\alpha}(\partial\Omega)$ and the pointwise product is continuous in $C^{0,\alpha}(\partial\Omega)$, we have $(\nu_\Omega)_l\mu$, $(\nu_\Omega)_r\mu\in C^{0,\alpha}(\partial\Omega)$. Then  a classical results on the single layer potential implies that 
\begin{eqnarray*}
&&v_\Omega^+[S_n,(\nu_\Omega)_l\mu],\quad v_\Omega^+[S_n,(\nu_\Omega)_r\mu]\in 
C^{1,\alpha}(\overline{\Omega})\,,
\\
&&v_\Omega^-[S_n,(\nu_\Omega)_l\mu]\,,\quad
v_\Omega^-[S_n,(\nu_\Omega)_r\mu]\in 
C^{1,\alpha}_{{\mathrm{loc}}}(\overline{\Omega^-})\,,
\end{eqnarray*}
(cf.~\textit{e.g.} \cite[Thm.~7.1 (i)]{DoLa17}). 
  Then formula (\ref{prop:vmlrmu1}) shows that
\[
v_\Omega^+[S_n,M_{lr}[\mu]]\in C^{0,\alpha}(\overline{\Omega})\,,\quad
v_\Omega^-[S_n,M_{lr}[\mu]]\in C^{0,\alpha}_{{\mathrm{loc}}}(\overline{\Omega^-})\,.
\]
Moreover, statement (i) and  
  the classical jump formulas for the partial derivatives of the single layer potential
imply the validity of statement (ii) (cf.~\textit{e.g.} \cite[Thm.~7.1 (ii)]{DoLa17}).

 (iii) By the definition of the distributional tangential  derivative $ M_{lr}$ and  by
 the classical jump formulas for the partial derivatives of the single layer potential
 (cf.~\textit{e.g.} \cite[Thm.~7.1 (ii)]{DoLa17}), we have
 \begin{eqnarray} \label{prop:vmlrmu3}
\lefteqn{
\langle M_{lr}[\mu],V_\Omega[S_n,\psi]\rangle
=-\langle \mu, M_{lr}[V_\Omega[S_n,\psi]]\rangle
}
\\ \nonumber
&&\quad 
=-\int_{\partial\Omega}\biggl(
(\nu_\Omega)_l(x)\frac{\partial}{\partial x_r}v_\Omega^+[S_n,\psi](x)
\\ \nonumber
&&\quad\quad 
-(\nu_\Omega)_r(x)\frac{\partial}{\partial x_l}v_\Omega^+[S_n, \psi](x)
\biggr)\mu(x)\,d\sigma_x
\\ \nonumber
&&\quad 
=-\int_{\partial\Omega}\biggl(
(\nu_\Omega)_l(x){\mathrm{p.v.}}\int_{\partial\Omega}
 \frac{\partial}{\partial x_r}S_n(x-y)\psi(y)\,d\sigma_y\mu(x) 
\\ \nonumber
&&\quad\quad 
-(\nu_\Omega)_r(x)
{\mathrm{p.v.}}\int_{\partial\Omega}
\frac{\partial}{\partial x_l}S_n(x-y) \psi(y)\,d\sigma_y
\mu(x)\biggr)\,d\sigma_x
\\ \nonumber
&&\quad 
=-\int_{\partial\Omega}
(\nu_\Omega)_l(x)\lim_{\epsilon\to 0}\int_{\partial\Omega\setminus {\mathbb{B}}_n(x,\epsilon)}
 \frac{\partial}{\partial x_r}S_n(x-y)\psi(y)\,d\sigma_y\mu(x)\,d\sigma_x
\\ \nonumber
&&\quad\quad 
+ \int_{\partial\Omega}(\nu_\Omega)_r(x)
\lim_{\epsilon\to 0}\int_{\partial\Omega\setminus {\mathbb{B}}_n(x,\epsilon)}
\frac{\partial}{\partial x_l}S_n(x-y) \psi(y)\,d\sigma_y
\mu(x)\,d\sigma_x\,.
\end{eqnarray}
 Now the inequality \cite[Thm.~3.3]{La23b} for the maximal function of the above principal values with $\mu$ equal to $1$ 
 and an  elementary inequality  on weakly singular integrals (cf.~\textit{e.g.},   \cite[Lem.~2.53]{DaLaMu21}) imply that
 \begin{eqnarray*} 
\lefteqn{
\sup_{x\in \partial\Omega}
\sup_{\epsilon\in ]0,+\infty[}
\left|\,
\int_{(\partial\Omega)\setminus{\mathbb{B}}_{n}(x,\epsilon)}
\frac{\partial S_n}{\partial x_h}(x-y)f(y)\,d\sigma_y
\right|
}
\\ \nonumber
&&\qquad 
\leq \sup_{x\in \partial\Omega}
\sup_{\epsilon\in ]0,+\infty[}
\left|\,
\int_{(\partial\Omega)\setminus{\mathbb{B}}_{n}(x,\epsilon)}
\frac{\partial S_n}{\partial x_h}(x-y)(f(y)-f(x))\,d\sigma_y
\right|
\\ \nonumber
&&\qquad\quad 
+ \sup_{\partial\Omega}|f| \sup_{x\in \partial\Omega}
\sup_{\epsilon\in ]0,+\infty[}
\left|\,
\int_{(\partial\Omega)\setminus{\mathbb{B}}_{n}(x,\epsilon)}
\frac{\partial S_n}{\partial x_h}(x-y)\,d\sigma_y
\right|
\\ \nonumber
&&\qquad 
\leq \sup_{x\in \partial\Omega}\frac{|f|_\alpha}{s_n}\int_{\partial\Omega}\frac{d\sigma_y}{|x-y|^{n-1-\alpha}}
\\ \nonumber
&&\qquad\quad 
+
\sup_{ \partial\Omega}|f|
\sup_{x\in \partial\Omega}
\sup_{\epsilon\in ]0,+\infty[}
\left|\,
\int_{(\partial\Omega)\setminus{\mathbb{B}}_{n}(x,\epsilon)}
\frac{\partial S_n}{\partial x_h}(x-y) \,d\sigma_y
\right|<+\infty
 \end{eqnarray*}
 for all $f\in C^{0,\alpha}(\partial\Omega)$ and $h\in\{1,\dots,n\}$. Then (\ref{prop:vmlrmu3}),  the Dominated Convergence Theorem, the Fubini Theorem in the set 
 \[
 \{
(x,y)\in   (\partial\Omega)^2:\,|x-y|\geq\epsilon\}
\]
 and statement (ii) imply that
 \begin{eqnarray*}  
\lefteqn{
\langle M_{lr}[\mu],V_\Omega[S_n,\psi]\rangle
=-\langle \mu, M_{lr}[V_\Omega[S_n,\psi]]\rangle
}
\\ \nonumber
&&\ 
=-\lim_{\epsilon\to 0}\int_{\partial\Omega}
(\nu_\Omega)_l(x) \int_{\partial\Omega\setminus {\mathbb{B}}_n(x,\epsilon)}
 \frac{\partial}{\partial x_r}S_n(x-y)\psi(y)\,d\sigma_y\mu(x)\,d\sigma_x
\\ \nonumber
&&\ \quad 
+ \lim_{\epsilon\to 0}\int_{\partial\Omega}(\nu_\Omega)_r(x)
 \int_{\partial\Omega\setminus {\mathbb{B}}_n(x,\epsilon)}
\frac{\partial}{\partial x_l}S_n(x-y) \psi(y)\,d\sigma_y
\mu(x)\,d\sigma_x
\\ \nonumber
&&\ 
=\lim_{\epsilon\to 0}\int_{\{
(x,y)\in   (\partial\Omega)^2:\,|x-y|\geq\epsilon\}}
(\nu_\Omega)_l(x) 
 \frac{\partial}{\partial y_r}\left(S_n(x-y)\right)
 \psi(y)\mu(x)
 \,d\sigma_y
 \,d\sigma_x
\\ \nonumber
&&\ \quad 
- \lim_{\epsilon\to 0}\int_{\{
(x,y)\in   (\partial\Omega)^2:\,|x-y|\geq\epsilon\}}(\nu_\Omega)_r(x)
\frac{\partial}{\partial y_l}\left(S_n(x-y) \right)
\psi(y)\mu(x)
\,d\sigma_y
\,d\sigma_x
\\ \nonumber
&&\ 
=\lim_{\epsilon\to 0}\int_{\{
(x,y)\in   (\partial\Omega)^2:\,|x-y|\geq\epsilon\}}
(\nu_\Omega)_l(x) 
 \frac{\partial}{\partial y_r}\left(S_n(x-y)\right)\mu(x)
 \psi(y)
 \,d\sigma_x\,d\sigma_y
\\ \nonumber
&&\ \quad 
- \lim_{\epsilon\to 0}\int_{\{
(x,y)\in   (\partial\Omega)^2:\,|x-y|\geq\epsilon\}}(\nu_\Omega)_r(x)
\frac{\partial}{\partial y_l}\left(S_n(x-y) \right)
\mu(x)
\psi(y)
\,d\sigma_x\,d\sigma_y
\\ \nonumber
&&\ 
=\lim_{\epsilon\to 0}\int_{\partial\Omega}
\int_{\partial\Omega\setminus {\mathbb{B}}_n(y,\epsilon)}
(\nu_\Omega)_l(x) 
 \frac{\partial}{\partial y_r}\left(S_n(x-y)\right)\mu(x)
  \,d\sigma_x
 \psi(y)
 \,d\sigma_y
\\ \nonumber
&&\ \quad 
- \lim_{\epsilon\to 0}\int_{\partial\Omega}
\int_{\partial\Omega\setminus {\mathbb{B}}_n(y,\epsilon)}
(\nu_\Omega)_r(x)
\frac{\partial}{\partial y_l}\left(S_n(x-y) \right)
\mu(x)
\,d\sigma_x
\psi(y)
\,d\sigma_y
\\ \nonumber
&&\ 
=\int_{\partial\Omega} \lim_{\epsilon\to 0}
\int_{\partial\Omega\setminus {\mathbb{B}}_n(y,\epsilon)}
(\nu_\Omega)_l(x) 
 \frac{\partial}{\partial y_r}\left(S_n(x-y)\right)\mu(x)
  \,d\sigma_x
\psi(y)
\,d\sigma_y
\\ \nonumber
&&\ \quad
-\int_{\partial\Omega} \lim_{\epsilon\to 0}
\int_{\partial\Omega\setminus {\mathbb{B}}_n(y,\epsilon)}
(\nu_\Omega)_r(x)
\frac{\partial}{\partial y_l}\left(S_n(x-y) \right)
\mu(x)
\,d\sigma_x
\psi(y)
\,d\sigma_y
\\ \nonumber
&&\  
=\int_{\partial\Omega} \frac{\partial}{\partial y_r}v_\Omega^+[S_n,(\nu_\Omega)_l\mu](y)\psi(y)
 -  \frac{\partial}{\partial y_l}v_\Omega^+[S_n,(\nu_\Omega)_r\mu](y)\psi(y)
\,d\sigma_y
\\ \nonumber
&&\
=\int_{\partial\Omega}
V_\Omega[S_n,M_{lr}[\mu]]
\psi(y)
\,d\sigma_y\,,
\end{eqnarray*}
and thus the proof is complete.
\hfill  $\Box$ 

\vspace{\baselineskip}

\begin{proposition}\label{prop:vmlrmuco}
 Let $\alpha\in]0,1[$. Let $\Omega$ be a bounded open   subset of ${\mathbb{R}}^{n}$ of class $C^{1,\alpha}$. Let $l,r\in \{1,\dots,n\}$. If $\mu\in C^{0,\alpha}(\partial\Omega)$, then $M_{lr}[\mu]$ belongs to $V^{-1,\alpha}(\partial\Omega)$. Moreover, the map  $M_{lr}$ from  $C^{0,\alpha}(\partial\Omega)$ to $V^{-1,\alpha}(\partial\Omega)$ that takes $\mu$ to $M_{lr}[\mu]$ is linear and continuous.
\end{proposition}
{\bf Proof.}  If $\mu\in C^{0,\alpha}(\partial\Omega)$, then Proposition \ref{prop:vmlrmu} (ii), (iii) and the equivalent definition  of  $V^{-1,\alpha}(\partial\Omega)$
in  \cite[Def.~13.2, 15.10, Thm.~18.1]{La24b})  imply that $M_{lr}[\mu]$ belongs to $V^{-1,\alpha}(\partial\Omega)$. Since $M_{lr}$ is linear, we now turn to prove that continuity of $M_{lr}$ by exploiting Theorem 18.3 (ii) of \cite{La24b}. To do so, we note that
\begin{eqnarray}\label{prop:vmlrmuco1} 
\lefteqn{
\left\|M_{lr}[\mu]\right\|_{V^{-1,\alpha}(\partial\Omega)}
}
\\ \nonumber
&& 
=\left\|V_\Omega\left[
M_{lr}[
\mu]-\frac{\langle M_{lr}[1],1\rangle}{\langle 1,1\rangle}1
 \right]
+\frac{\langle M_{lr}[1],1\rangle}{\langle 1,1\rangle}1
 \right\|_{C^{0,\alpha}(\partial\Omega)}
 \\ \nonumber
&& 
=\left\|V_\Omega\left[
M_{lr}[
\mu]\right] \right\|_{C^{0,\alpha}(\partial\Omega)}
 \\ \nonumber
&& 
=\left\|
\frac{\partial}{\partial x_r}v_\Omega^+[S_n,(\nu_\Omega)_l\mu] 
-\frac{\partial}{\partial x_l}v_\Omega^+[S_n,(\nu_\Omega)_r\mu]
 \right\|_{C^{0,\alpha}(\partial\Omega)}
 \quad\forall \mu\in C^{0,\alpha}(\partial\Omega)\,.
\end{eqnarray}
On the other hand the components of $\nu_\Omega$  belong to $C^{0,\alpha}(\partial\Omega)$ and  the pointwise product is continuous in $C^{0,\alpha}(\partial\Omega)$ and $v_\Omega^{+}[S_n,\cdot]$ is linear and continuous from
$C^{0,\alpha}(\partial\Omega)$ to $C^{1,\alpha}(\overline{\Omega})$ (cf.~\textit{e.g.}, \cite[Thm.~4.25]{DaLaMu21}). Hence, there exists a constant $c\in]0,+\infty[$ such that 
\begin{eqnarray}\label{prop:vmlrmuco2} 
\lefteqn{
\left\|
\frac{\partial}{\partial x_r}v_\Omega^+[S_n,(\nu_\Omega)_l\mu] 
-\frac{\partial}{\partial x_l}v_\Omega^+[S_n,(\nu_\Omega)_r\mu]
 \right\|_{C^{0,\alpha}(\partial\Omega)}
}
\\ \nonumber
&&\qquad\qquad\qquad
\leq	{\color{black}c} 	\|\mu\|_{C^{0,\alpha}(\partial\Omega)}
\qquad\forall \mu\in C^{0,\alpha}(\partial\Omega)\,.
\end{eqnarray}
Then \cite[Thm.~18.3 (ii)]{La24b} and inequalities (\ref{prop:vmlrmuco1}), (\ref{prop:vmlrmuco2}) imply that the linear operator $M_{lr}$ is continuous.\hfill  $\Box$ 

\vspace{\baselineskip}

Next, we note that if $n=3$ and $\mu$ is a H\"{o}lder continuous function on the boundary of $\Omega$, then   the tangential  gradient of $\mu$ has not been defined and thus we do not know what  $\nu_\Omega\bar{\wedge}{\mathrm{grad}}_{\partial\Omega}\mu$ means. 
We now show that there is a canonical way to extend the map $\nu_\Omega\bar{\wedge}  {\mathrm{grad}}_{\partial\Omega}\mu$ of the variable $\mu$ from the space of $C^1$ functions to the whole space of H\"{o}lder continuous functions on the boundary.
\begin{theorem}[of the tangential component]\label{thm:tanpargr}
 Let $\alpha\in]0,1[$. Let $\Omega$ be a bounded open set of ${\mathbb{R}}^3$ of class $C^{1,\alpha}$. There exists one and only one linear and continuous map
 $T_\Omega$  from $C^{0,\alpha}(\partial\Omega)$ to
 $V^{-1,\alpha}(\partial\Omega, {\mathbb{C}}^3)$ such that the following two conditions are satisfied.
 \begin{enumerate}
\item[(i)] The following equality holds
\begin{equation}\label{thm:tanpargr1}
T_\Omega[\mu]=\nu_\Omega\bar{\wedge}{\mathrm{grad}}_{\partial\Omega}\mu\qquad\forall \mu\in C^{1}(\partial\Omega)\,.
\end{equation}
\item[(ii)] If $\beta\in]0,\alpha[$, then $T_\Omega$ is continuous from 
$C^{0,\alpha}(\partial\Omega)$ with the norm of
$C^{0,\beta}(\partial\Omega)$ to $\left(C^{1,\alpha}(\partial\Omega, {\mathbb{C}}^3)\right)'$ with the weak$^*$ topology.
\end{enumerate}
Moreover,
\begin{equation}\label{thm:tanpargr2}
T_\Omega[\mu]=
M_{23}[\mu]e_1
-M_{13}[\mu] e_2+M_{12}[\mu] e_3\qquad \text{a.a.} \ \text{on}\ \partial\Omega\,,
\end{equation}
for all $\mu\in C^{0,\alpha}(\partial\Omega)$,
where $\{e_1,e_2,e_3\}$ is the canonical basis of ${\mathbb{R}}^3$. 
\end{theorem}
 {\bf Proof.} Let $T_\Omega$ be the operator from $C^{0,\alpha}(\partial\Omega)$ to
 $V^{-1,\alpha}(\partial\Omega)$ that is defined by equality (\ref{thm:tanpargr2}).  By Proposition \ref{prop:vmlrmuco}, we know that $T_\Omega$ is linear and continuous.  By 
 Lemma \ref{lem:taden=3}, $T_\Omega$ satisfies equality (\ref{thm:tanpargr1}) of condition (i). If $\beta\in ]0,\alpha[$, then $\Omega$ is of class $C^{1,\beta}$ and thus Proposition \ref{prop:vmlrmuco} implies that
 the right hand side of equality (\ref{thm:tanpargr2})  defines a linear and continuous map 
 $T_{\Omega,\beta}$ from 
 $C^{1,\beta}(\partial\Omega)$ to
 $V^{-1,\beta}(\partial\Omega)$, which is continuously embedded into $\left(C^{1,\beta}(\partial\Omega, {\mathbb{C}}^3)\right)'$ and accordingly into $\left(C^{1,\beta}(\partial\Omega, {\mathbb{C}}^3)\right)'$ with the weak$^*$ topology. By definition of $T_{\Omega,\beta}$, we have 
 \[
\langle T_\Omega[\mu],v\rangle=\langle T_{\Omega,\beta}[\mu],v\rangle\qquad\forall v\in C^{1,\alpha}(\partial\Omega, {\mathbb{C}}^3)
\subseteq C^{1,\beta}(\partial\Omega, {\mathbb{C}}^3)\,,
\]
 for all $\mu\in C^{1,\alpha}(\partial\Omega)\left(\subseteq
 C^{1, \beta}(\partial\Omega)
 \right)$. Thus  if $v\in C^{1,\alpha}(\partial\Omega, {\mathbb{C}}^3)$, the map $\langle T_{\Omega,\beta}[\cdot],v\rangle$ is continuous from $C^{1,\beta}(\partial\Omega)$  to ${\mathbb{C}}$ and   
 \[
 \langle T_{\Omega} [\cdot],v\rangle= \langle T_{\Omega,\beta}[\cdot]_{|C^{1,\alpha}(\partial\Omega, {\mathbb{C}}^3)},v\rangle
 \]
   is continuous from $C^{1,\alpha}(\partial\Omega)$ with the norm of $C^{1,\beta}(\partial\Omega)$ to ${\mathbb{C}}$ and condition (ii) holds true. 
   
We now assume that the linear operator $\tilde{T}_\Omega$  from $C^{0,\alpha}(\partial\Omega)$ to
 $V^{-1,\alpha}(\partial\Omega)$ satisfies conditions (i) and (ii). Then we assume that  $\mu\in 
 C^{0,\alpha}(\partial\Omega)$ and we turn to prove that
 $\tilde{T}_\Omega[\mu]=\tilde{T}_\Omega[\mu]$.  By the  known approximation Lemma \ref{lem:aprbdma}, there exists a sequence
$\{\mu_j\}_{j\in {\mathbb{N}}}$ in $C^{1,\alpha}(\partial\Omega)$ such that
\[
\sup_{j\in {\mathbb{N}}}\|\mu_j\|_{C^{0,\alpha}(\partial\Omega)}<+\infty\,,
\qquad
\lim_{j\to\infty}\mu_j=\mu\quad\text{in}\ C^{0,\beta}(\partial\Omega)\quad\forall \beta\in]0,\alpha[\,.
\]
 Now let $v\in C^{1,\alpha}(\partial\Omega, {\mathbb{C}}^3)$. Since $\tilde{T}_\Omega$ satisfies condition (ii), we  have
\begin{equation}\label{thm:tanpargr3}
\langle \tilde{T}_\Omega[\mu],v\rangle=\lim_{j\to\infty}\langle \tilde{T}_\Omega[\mu_j],v\rangle\,.
\end{equation}
Then condition (i) and Lemma \ref{lem:taden=3} imply that
\begin{eqnarray}\label{thm:tanpargr4}
\lefteqn{
\lim_{j\to\infty}\langle \tilde{T}_\Omega[\mu_j],v\rangle
=\lim_{j\to\infty}\langle \nu_\Omega\bar{\wedge}{\mathrm{grad}}_{\partial\Omega}\mu_j,v\rangle
}
\\ \nonumber
&&\qquad 
=\lim_{j\to\infty}\langle M_{23}[\mu_j]e_1
-M_{13}[\mu_j] e_2+M_{12}[\mu_j] e_3,v\rangle
\end{eqnarray} 
 and thus Proposition \ref{prop:vmlrmuco} implies that 
\begin{eqnarray}\label{thm:tanpargr5}
\lefteqn{
\lim_{j\to\infty}\langle M_{23}[\mu_j]e_1
-M_{13}[\mu_j] e_2+M_{12}[\mu_j] e_3,v\rangle
}
\\ \nonumber
&&\qquad
= \langle M_{23}[\mu]e_1
-M_{13}[\mu] e_2+M_{12}[\mu] e_3,v\rangle
=\langle T_\Omega[\mu],v\rangle\,.
\end{eqnarray}
Hence, the limiting relations (\ref{thm:tanpargr3})--(\ref{thm:tanpargr5}) imply that
\[
\langle \tilde{T}_\Omega[\mu],v\rangle=\langle T_\Omega[\mu],v\rangle
\qquad\forall v\in  C^{1,\alpha}(\partial\Omega, {\mathbb{C}}^3)
\]
and accordingly that  $\tilde{T}_\Omega[\mu]=T_\Omega[\mu]$. Then  the proof is complete.\hfill  $\Box$ 

\vspace{\baselineskip}

By Theorem \ref{thm:tanpargr}, $T_\Omega[\mu]$ coincides with  $\nu_\Omega\bar{\wedge}{\mathrm{grad}}_{\partial\Omega}\mu$ if $\mu$ is of class $C^1$. By abuse of notation $T_\Omega[\mu]$ is sometimes written as $\nu_\Omega\bar{\wedge}{\mathrm{grad}}_{\partial\Omega}\mu$ or even as $\nu_\Omega\bar{\wedge}\nabla\mu$ also in case $\mu$ is of class $C^{0,\alpha}$. Then we can deduce the validity of the following immediate corollary of Theorem \ref{thm:tanpargr}. 
\begin{corollary}\label{cor:nuwgrfc}
 Let $m\in {\mathbb{N}}$. Let $\alpha\in]0,1[$. Let $\Omega$ be a bounded open set of ${\mathbb{R}}^3$ of class $C^{\max\{1,m\},\alpha}$. Then the map from 
  $C^{m,\alpha}(\partial\Omega)$ to $V^{m-1,\alpha}(\partial\Omega, {\mathbb{C}}^3)$ that takes $\mu$ to $T_\Omega[\mu]$ is linear and continuous (see (\ref{eq:vm-1a}) for the definition of $V^{m-1,\alpha}(\partial\Omega)$).
\end{corollary}
{\bf Proof.} If $m=0$, the statement follows by  Theorem \ref{thm:tanpargr}. If instead $m\geq 1$, then the statement follows by Theorem \ref{thm:tanpargr} (i), by
Lemma \ref{lem:taden=3} and by the known continuity of the tangential derivatives $M_{lj}$ from $C^{m,\alpha}(\partial\Omega)$ to $C^{m-1,\alpha}(\partial\Omega)$ for all $l,j\in\{1,2,3\}$.\hfill  $\Box$ 

\vspace{\baselineskip}

We now turn to introduce the tangential divergence of a distributional tangential vector field.
\begin{definition}\label{defn:tandiv}
 Let $\Omega$ be a bounded open   subset of ${\mathbb{R}}^{3}$ of class $C^{1}$. 
 Let $a\in C^0(\partial\Omega,  {\mathbb{C}}^3)$.   Let
 \[
 a (x)\cdot\nu_\Omega(x)=0\qquad \forall x\in \partial\Omega\,.
 \]
Then the weak tangential divergence ${\mathrm{div}}_{\partial\Omega}a$ of $a$ is the (distribution) of $\left(C^1(\partial\Omega  ) \right)'$ that is defined by the equality
\begin{equation}\label{defn:tandiv1}
\langle {\mathrm{div}}_{\partial\Omega}a,\varphi\rangle=
-\int_{\partial\Omega}{\mathrm{grad}}_{\partial\Omega}\varphi(x)\cdot a(x)\,d\sigma_x\qquad\forall \varphi\in C^1(\partial\Omega  )	\,,
\end{equation}
(see (\ref{defn:tangrad}) for the definition of tangential gradient.)
\end{definition}
Then we prove the following continuity statement for the tangential divergence that extends the corresponding classical result in the sense that it applies to  H\"{o}lder continuous functions, with no assumptions on their tangential derivatives.
\begin{proposition}\label{prop:tandivco}
 Let $\alpha\in]0,1[$. Let $\Omega$ be a bounded open   subset of ${\mathbb{R}}^{n}$ of class $C^{1,\alpha}$. Then the map ${\mathrm{div}}_{\partial\Omega}$ is linear and continuous from the closed subspace 
 \[
 C^{0,\alpha}_t(\partial\Omega,{\mathbb{C}}^n)=\left\{
 a\in C^{0,\alpha}(\partial\Omega,  {\mathbb{C}}^n):\,  a(x)\cdot\nu_\Omega(x)=0\,,\ 
   \forall x\in \partial\Omega
 \right\}
\]
of $ C^{0,\alpha}_t(\partial\Omega,{\mathbb{C}}^n)$
to $V^{-1,\alpha}(\partial\Omega)$. Moreover,
\begin{equation}\label{prop:tandivco1}
{\mathrm{div}}_{\partial\Omega}a=\sum_{l,r=1}^nM_{lr}[a_r(\nu_\Omega)_l]\qquad\forall
a\in C^{0,\alpha}_t(\partial\Omega,{\mathbb{C}}^n)\,.
\end{equation}
\end{proposition}
{\bf Proof.} Since the components of $\nu_\Omega$ belong to $C^{0,\alpha}(\partial\Omega)$ and the pointwise product is continuous in $C^{0,\alpha}(\partial\Omega)$, the space $C^{0,\alpha}_t(\partial\Omega,{\mathbb{C}}^n)$ is closed in $C^{0,\alpha}(\partial\Omega,{\mathbb{C}}^n)$ and 
the map from 
$C^{0,\alpha}_t(\partial\Omega,{\mathbb{C}}^n)$ to $C^{0,\alpha}(\partial\Omega)$ that takes $a$ to $a_r(\nu_\Omega)_l$ is linear and continuous for all $r$, $l\in \{1,\dots,n\}$. Hence, Proposition \ref{prop:vmlrmuco} implies that the right hand side of formula (\ref{prop:tandivco1}) defines a linear and continuous map from 
$C^{0,\alpha}_t(\partial\Omega,{\mathbb{C}}^n)$ to $V^{-1,\alpha}(\partial\Omega)$. Hence, it suffices to show the validity of formula (\ref{prop:tandivco1}).  To do so, we note that
 \begin{eqnarray}\label{prop:tandivco2}
 \lefteqn{
\langle {\mathrm{div}}_{\partial\Omega}a,\varphi\rangle=
-\int_{\partial\Omega}a\cdot {\mathrm{grad}}_{\partial\Omega}\varphi 	\,d\sigma 
}
\\ \nonumber
&&\qquad\qquad
= -\sum_{l,r=1}^n\int_{\partial\Omega}a_rM_{lr}[\varphi](\nu_\Omega)_l\,d\sigma
=-\sum_{l,r=1}^n\langle a_r(\nu_\Omega)_l,M_{lr}[\varphi]\rangle
 \\ \nonumber
&&\qquad\qquad
=\sum_{l,r=1}^n\langle  M_{lr}[ a_r(\nu_\Omega)_l],\varphi\rangle
 \qquad\forall \varphi\in C^{1,\alpha}(\partial\Omega)\,,
\end{eqnarray}
(cf. (\ref{defn:tangrad}), (\ref{defn:tanderdis1})).\hfill  $\Box$ 

\vspace{\baselineskip}
Next we introduce the following distributional  form of a known classical technical statement that we need later (cf.~Kirsch and Hettlich \cite[Cor.~A.20]{KiHe15}).
\begin{lemma}\label{lem:nucuv}
 Let $\alpha\in]0,1[$. Let $\Omega$ be a bounded open subset of ${\mathbb{R}}^3$ of class $C^{1,\alpha}$. Let $a\in C^{0,\alpha}(\overline{\Omega},{\mathbb{C}}^3)$,  then
 \begin{equation}\label{lem:nucuv1}
 -{\mathrm{div}}_{\partial\Omega} (\nu_\Omega\bar{\wedge}a)=
 M_{23}[a_1]-M_{13}[a_2]+M_{12}[a_3] \qquad\text{on}\ \partial\Omega\,.
 \end{equation}
 If we further assume that  ${\mathrm{curl}}\,a\in C^{0,\alpha}(\overline{\Omega},{\mathbb{C}}^3)$, then we also have
 \begin{equation}\label{lem:nucuv1b}
 M_{23}[a_1]-M_{13}[a_2]+M_{12}[a_3]=\nu_\Omega\cdot {\mathrm{curl}}\,a\quad\text{on}\ \partial\Omega\,.
 \end{equation}
 \end{lemma}
{\bf Proof.} By Definition \ref{defn:tandiv} of distributional tangential divergence, by  Lemma  \ref{lem:taden=3}
 on the tangential component of the gradient and Definition \ref{defn:tanderdis} of tangential derivative of a distribution, we have
\begin{eqnarray*} 
\lefteqn{
-\langle {\mathrm{div}}_{\partial\Omega} (\nu_\Omega\bar{\wedge}a),\varphi\rangle
}
\\ \nonumber
&&\qquad 
=\int_{\partial\Omega} (\nu_\Omega\bar{\wedge}a)\cdot {\mathrm{grad}}_{\partial\Omega}\varphi 
\,d\sigma
=-\int_{\partial\Omega} (\nu_\Omega\bar{\wedge}
{\mathrm{grad}}_{\partial\Omega}\varphi 
)\cdot a
\,d\sigma
 \\ \nonumber
&&\qquad 
=-\int_{\partial\Omega}(M_{23}[\phi]e_1
-M_{13}[\phi] e_2+M_{12}[\phi] e_3)\cdot a \,d\sigma
\\ \nonumber
&&\qquad 
=\langle M_{23}[a_1] 
-M_{13}[a_2] +M_{12}[a_3]  ,\varphi \rangle 
\qquad\forall\varphi\in C^{1,\alpha}(\partial\Omega)\,.
\end{eqnarray*}
 and thus the equality of (\ref{lem:nucuv1}) holds true. We now further assume that  ${\mathrm{curl}}\,a\in C^{0,\alpha}(\overline{\Omega},{\mathbb{C}}^3)$ and turn to prove equality (\ref{lem:nucuv1b}). By Proposition \ref{prop:vmlrmuco}, the left hand side of  equality   (\ref{lem:nucuv1b}) belongs  to the space
$V^{-1,\alpha}(\partial\Omega,{\mathbb{C}}^3)$. By our assumptions on $a$, by the membership in $C^{0,\alpha}(\partial\Omega)$ of the components of $\nu_\Omega$ and by the continuity of the pointwise product in $C^{0,\alpha}(\partial\Omega)$, we know that the right hand side of   equality    (\ref{lem:nucuv1b}) belongs to
$C^{0,\alpha}(\partial\Omega)$, which in turn is contained in $V^{-1,\alpha}(\partial\Omega,{\mathbb{C}}^3)$. By \cite[Thm.~17.5]{La24b},   equality  (\ref{lem:nucuv1b}) follows provided that we can prove that
\begin{eqnarray}\label{lem:nucuv2}
\lefteqn{\langle  
 M_{23}[a_1](y)-M_{13}[a_2](y)+M_{12}[a_3](y) , S_3(x-y)\rangle
}
\\ \nonumber
&&
\qquad\qquad\qquad=\int_{\partial\Omega}S_3(x-y)
\nu_\Omega(y)\cdot {\mathrm{curl}}\,a(y)\,d\sigma \qquad\forall x\in\Omega\,.
\end{eqnarray}
Indeed, the single layer of an element of $V^{-1,\alpha}(\partial\Omega)$ has a continuous extension from $\Omega$ to $\overline{\Omega}$ (cf.  
\cite[Defn.~15.10, Thm.~18.1]{La24b}). To do so, we now fix $x\in \Omega$, we take a function $\phi_x\in C^\infty_c({\mathbb{R}}^n)$ that vanishes in an open neighborhood of $x$ and that equals $1$ in an open neighborhood of $\partial\Omega$. Then we set
\[
\psi_x(y)\equiv \phi_x (y)S_n(x-y)\quad\forall y\in {\mathbb{R}}^n\setminus\{x\}\,,
\quad
\psi_x(x)\equiv 0\,.
\]
Clearly, $\psi_x\in C^\infty_c({\mathbb{R}}^n)$ and it suffices to show that
\begin{eqnarray}\label{lem:nucuv3}
\langle
 M_{23}[a_1] -M_{13}[a_2] +M_{12}[a_3]  , \psi_x \rangle
=\int_{\partial\Omega}(\nu_\Omega \cdot {\mathrm{curl}}\,a)\psi_x\,d\sigma\,.
\end{eqnarray}
We first note that 
 \begin{eqnarray}\label{lem:nucuv4}
\lefteqn{M_{23}[b_1] -M_{13}[b_2] +M_{12}[b_3]
}
\\ \nonumber
&& 
=(\nu_\Omega)_2\frac{\partial b_1}{\partial x_3}-(\nu_\Omega)_3\frac{\partial b_1}{\partial x_2}
-\left(
(\nu_\Omega)_1\frac{\partial b_2}{\partial x_3}-(\nu_\Omega)_3\frac{\partial b_2}{\partial x_1}
\right)
\\ \nonumber
&&\qquad\qquad\qquad\qquad 
+(\nu_\Omega)_1\frac{\partial b_3}{\partial x_2}-(\nu_\Omega)_2\frac{\partial b_3}{\partial x_1}
\\ \nonumber
&& 
=(\nu_\Omega)_1 \left[\frac{\partial b_3}{\partial x_2}-\frac{\partial b_2}{\partial x_3}\right]
+(\nu_\Omega)_2 \left[\frac{\partial b_1}{\partial x_3}-\frac{\partial b_3}{\partial x_1}\right]
+(\nu_\Omega)_3\left[\frac{\partial b_2}{\partial x_1}-\frac{\partial b_1}{\partial x_2}\right]
\\ \nonumber
&&= \nu_\Omega \cdot {\mathrm{curl}}\,b
\qquad\forall b\in C^\infty(\overline{\Omega},{\mathbb{C}}^3)\,.
\end{eqnarray}
Since $\psi_x$ is of class $C^\infty$ on $\overline{\Omega}$, our assumption on 
$ {\mathrm{curl}}\,a $ implies that $\psi_x{\mathrm{curl}}\,a  $ belongs to  $C^{0,\alpha}(\overline{\Omega},{\mathbb{C}}^3)$ and accordingly
 the Divergence Theorem \ref{thm:div0a} in distributional form 
 and the known equality
 \[
 {\mathrm{div}}\,\left(\psi_x{\mathrm{curl}}\,a\right)
 =(\nabla \psi_x)\cdot {\mathrm{curl}}\,a +\psi_x{\mathrm{div}}\,{\mathrm{curl}}\,a
 =(\nabla \psi_x)\cdot {\mathrm{curl}}\,a \qquad\text{in}\ \Omega
 \]
 imply that
 \begin{eqnarray}\label{lem:nucuv5}
\lefteqn{ 
 \int_{\partial\Omega}(\nu_\Omega \cdot {\mathrm{curl}}\,a)\psi_x\,d\sigma
=\int_{\partial\Omega} \nu_\Omega \cdot \left(\psi_x{\mathrm{curl}}\,a \right) \,d\sigma
}
\\ \nonumber
&&\qquad
=\langle E^\sharp_\Omega[{\mathrm{div}}\,\left(\psi_x{\mathrm{curl}}\,a\right) ],1\rangle
=\langle E^\sharp_\Omega[(\nabla \psi_x)\cdot {\mathrm{curl}}\,a ],1\rangle\,.
\end{eqnarray}
Let $\beta\in]0,\alpha[$.
By Lemma \ref{lem:apr1a}, there exists a sequence $\{a_j\}_{j\in {\mathbb{N}}}$ in 
 $C^{\infty}(\overline{\Omega}, {\mathbb{C}}^3)$ such that
 \begin{equation}\label{lem:nucuv6a}
 \sup_{j\in {\mathbb{N}}}\|a_j\|_{
 C^{0,\alpha}(\overline{\Omega}) 
 }<+\infty\,,\qquad
 \lim_{j\to\infty}a_j=a\quad\text{in}\ C^{0,\beta}(\overline{\Omega})  \,.
 \end{equation}
 Then we also have
 \[
 \lim_{j\to\infty} {\mathrm{curl}}\,a_j= {\mathrm{curl}}\,a\quad\text{in}\ C^{-1,\beta}(\overline{\Omega})\,.
 \]
Since the multiplication by a function of class $C^{\infty}(\overline{\Omega})$ is continuous in 
 $C^{-1,\beta}(\overline{\Omega})$ (cf. \cite[Lem.~2.3]{La25}), $E^\sharp_\Omega$ is continuous from $C^{-1,\beta}(\overline{\Omega})$ to $\left(C^{1,\beta}(\overline{\Omega})\right)'$ (cf.~Proposition \ref{prop:nschext}) and the natural duality pairing of $C^{1,\beta}(\overline{\Omega})$ is continuous, 
 equalities (\ref{prop:nschext3}), (\ref{lem:nucuv4}), (\ref{lem:nucuv5}) and the Divergence  Theorem  imply that
  \begin{eqnarray}\label{lem:nucuv6}
\lefteqn{\int_{\partial\Omega}(\nu_\Omega \cdot {\mathrm{curl}}\,a)\psi_x\,d\sigma
=\langle E^\sharp_\Omega[(\nabla \psi_x)\cdot {\mathrm{curl}}\,a ],1\rangle
}
\\ \nonumber
&&\qquad 
=\lim_{j\to\infty}\langle E^\sharp_\Omega[(\nabla \psi_x)\cdot {\mathrm{curl}}\,a_j ],1\rangle
\\ \nonumber
&&\qquad 
=\lim_{j\to\infty}\int_\Omega  (\nabla \psi_x)\cdot {\mathrm{curl}}\,a_j \,dx
\\ \nonumber
&&\qquad 
=\lim_{j\to\infty}\int_\Omega  
{\mathrm{div}}\,\left(\psi_x{\mathrm{curl}}\,a_j\right)\,dx
\\ \nonumber
&&\qquad 
=\lim_{j\to\infty}
\int_{\partial\Omega}(\nu_\Omega \cdot {\mathrm{curl}}\,a_j)\psi_x\,d\sigma
\\ \nonumber
&&\qquad 
=\lim_{j\to\infty}
\int_{\partial\Omega}
M_{23}[(a_j)_1] -M_{13}[(a_j)_2] +M_{12}[(a_j)_3]\psi_x\,d\sigma
\\ \nonumber
&&\qquad 
=\lim_{j\to\infty}
\langle
M_{23}[(a_j)_1] -M_{13}[(a_j)_2] +M_{12}[(a_j)_3],\psi_x\rangle
\,.
\end{eqnarray}
 Then Proposition \ref{prop:vmlrmuco}, implies that
  \begin{eqnarray}\label{lem:nucuv7}
\lefteqn{\lim_{j\to\infty}
\langle
M_{23}[(a_j)_1] -M_{13}[(a_j)_2] +M_{12}[(a_j)_3],\psi_x\rangle
}
\\ \nonumber
&&\qquad\qquad\qquad\qquad 
=\langle
M_{23}[a_1] -M_{13}[a_2] +M_{12}[a_3],\psi_x\rangle
\end{eqnarray}
 and thus equalities (\ref{lem:nucuv6}) and (\ref{lem:nucuv7}) imply the validity of equality (\ref{lem:nucuv3}) and thus the proof is complete.\hfill  $\Box$ 

\vspace{\baselineskip}

  \section{Preliminaries to the acoustic volume potential}\label{sec:acvolpot}

Let $\alpha\in]0,1]$ and $m\in {\mathbb{N}}$.  If $\Omega$ is a  bounded open subset of ${\mathbb{R}}^{n}$, then we can consider the restriction map $r_{|\overline{\Omega}}$ from ${\mathcal{D}}({\mathbb{R}}^n)$ to $C^{m,\alpha}(\overline{\Omega})$. Then the transpose map $r_{|\overline{\Omega}}^t$ is linear and continuous from $(C^{m,\alpha}(\overline{\Omega}))'$ to ${\mathcal{D}}'({\mathbb{R}}^n)$. Moreover, if $\mu\in (C^{m,\alpha}(\overline{\Omega}))'$, then $r_{|\overline{\Omega}}^t\mu$ has compact support. Hence, it makes sense to consider the convolution of 
  $r_{|\overline{\Omega}}^t\mu$ with  the fundamental solution of either the Laplace or the Helmholtz  operator. Thus we are now ready to introduce the following known definition.
 \begin{definition}\label{defn:dvpsl}
 Let $\alpha\in]0,1]$, $m\in {\mathbb{N}}$. 
 Let   $\Omega$ be a bounded open subset of ${\mathbb{R}}^{n}$. Let $\lambda\in {\mathbb{C}}$.  Let $S_{n,\lambda} $ be a fundamental solution 
of the operator $\Delta+\lambda$. If $\mu\in (C^{m,\alpha}(\overline{\Omega}))'$, then the (distributional) volume potential relative to $S_{n,\lambda} $ and $\mu$ is the distribution
\[
{\mathcal{P}}_\Omega[S_{n,\lambda} ,\mu]=(r_{|\overline{\Omega}}^t\mu)\ast S_{n,\lambda}  \in {\mathcal{D}}'({\mathbb{R}}^n)\,.
\]
\end{definition}
Under the assumptions of Definition \ref{defn:dvpsl}, we  set
\begin{eqnarray}\label{prop:dvpsl3}
{\mathcal{P}}_\Omega^+[S_{n,\lambda} ,\mu]&\equiv&\left((r_{|\overline{\Omega}}^t\mu)\ast S_{n,\lambda}  \right)_{|\Omega}
\qquad\text{in}\ \Omega\,,
\\ \nonumber
{\mathcal{P}}_\Omega^-[S_{n,\lambda} ,\mu]  &\equiv&
\left((r_{|\overline{\Omega}}^t\mu)\ast S_{n,\lambda}  \right)_{|\Omega^-}
\qquad\text{in}\ \Omega^-\,.
\end{eqnarray}
In general, $(r_{|\overline{\Omega}}^t\mu)\ast S_{n,\lambda} $ is not a function, \textit{i.e.} $(r_{|\overline{\Omega}}^t\mu)\ast S_{n,\lambda} $ is not a distribution that is associated to a locally integrable function in ${\mathbb{R}}^n$. 
However, this is the case if for example  $\mu$ is associated to a function $f$ of $ L^\infty(\Omega)$,  and thus 
 the (distributional) volume potential relative to $S_{n,\lambda} $ and $\mu$ is associated to the function
\begin{equation}\label{prop:dvpsa1}
\int_{\Omega}S_{n,\lambda} (x-y)f(y)\,dy\qquad{\mathrm{a.a.}}\ x\in {\mathbb{R}}^n\,,
\end{equation}
that is locally integrable in ${\mathbb{R}}^n$  and that with some abuse of notation we  denote by the symbol    ${\mathcal{P}}_\Omega[S_{n,\lambda} , f]$.  Then the following classical result is known (cf.~\textit{e.g.}, \cite[ Thm.~20]{La24d}).
\begin{theorem}\label{thm:nwtdma} 
 Let $m\in {\mathbb{N}}$, $\alpha\in]0,1[$. Let $\Omega$ be a bounded open subset of  ${\mathbb{R}}^n$ of class $C^{m+1,\alpha}$. Let $\lambda\in {\mathbb{C}}$.  Let $S_{n,\lambda}  $ be a fundamental solution 
of the operator $\Delta+\lambda$. Then the following statements hold. 
 
 \item[(i)]  ${\mathcal{P}}_\Omega^+[S_{n,\lambda} ,\cdot]$ is linear and continuous from $C^{m,\alpha}(\overline{\Omega})$ to $C^{m+2,\alpha}(\overline{\Omega})$.
\item[(ii)]   ${\mathcal{P}}_\Omega^-[S_{n,\lambda} ,\cdot]$ is linear and continuous from $C^{m,\alpha}(\overline{\Omega})$ to   $C^{m+2,\alpha}(\overline{{\mathbb{B}}_n(0,r)}\setminus\Omega)$ for all $r\in]0,+\infty[$ such that $\overline{\Omega}\subseteq {\mathbb{B}}_n(0,r)$.
\end{theorem}
 Instead, for a Schauder space of negative exponent, the following statement holds (cf.~\textit{e.g.}, 
 \cite[Prop.~23]{La24d}), that is a  generalization to volume potentials of  nonhomogeneous second order elliptic operators of a known result for the Laplace operator  (cf.~\cite[Thm.~3.6 (ii)]{La08a}, Dalla Riva, the author and Musolino~\cite[Thm.~7.19]{DaLaMu21}).
    \begin{proposition}\label{prop:dvpsnecr-1a}
  Let $\alpha\in]0,1[$.    Let $\Omega$ be a bounded open subset of ${\mathbb{R}}^{n}$ of class $C^{1,\alpha}$. Let $\lambda\in {\mathbb{C}}$.  Let $S_{n,\lambda}  $ be a fundamental solution 
of the operator $\Delta+\lambda$.   Then the following statements hold. 
 \begin{enumerate}
\item[(i)] If  $f=  f_{0}+\sum_{j=1}^{n}\frac{\partial}{\partial x_{j}}f_{j}\in C^{-1,\alpha}(\overline{\Omega}) $, then
 \begin{equation}\label{prop:dvpsnecr-1a2}
{\mathcal{P}}_\Omega^+[S_{n,\lambda} ,E^\sharp[f]]\in C^{1,\alpha}(\overline{\Omega}), \ 
{\mathcal{P}}_\Omega^-[S_{n,\lambda} ,E^\sharp[f]]\in C^{1,\alpha}_{{\mathrm{loc}} }(\overline{\Omega^-}) \,,\end{equation}
and 
\begin{equation}\label{defn:Ppm1}
{\mathcal{P}}_\Omega^+[S_{n,\lambda} ,E^\sharp[f]](x)={\mathcal{P}}_\Omega^-[S_{n,\lambda} ,E^\sharp[f]](x)\qquad\forall x\in\partial\Omega\,.
\end{equation}
Moreover,
\begin{eqnarray}\label{prop:dvpsnecr-1a2a}
&&\Delta {\mathcal{P}}_\Omega^+[S_{n,\lambda} ,E^\sharp[f]] 
+\lambda{\mathcal{P}}_\Omega^+[S_{n,\lambda} ,E^\sharp[f]]= f\qquad\textit{in}\ {\mathcal{D}}'(\Omega)\,,
\\ \nonumber
&&\Delta  {\mathcal{P}}_\Omega^-[S_{n,\lambda} ,E^\sharp[f]]
+\lambda  {\mathcal{P}}_\Omega^-[S_{n,\lambda} ,E^\sharp[f]]= 0\qquad\textit{in}\ {\mathcal{D}}'({\mathbb{R}}^n\setminus\overline{\Omega})
\,.
\end{eqnarray}
\item[(ii)] The linear operator  ${\mathcal{P}}_\Omega^+[S_{n,\lambda} ,E^\sharp[\cdot]]$ is  continuous from $C^{-1,\alpha}(\overline{\Omega}) $ to $C^{1,\alpha}(\overline{\Omega})$.
\item[(iii)] Let $r\in ]0,+\infty[$ be such that $\overline{\Omega}\subseteq {\mathbb{B}}_n(0,r)$. Then  the linear  operator ${\mathcal{P}}_\Omega^-[S_{n,\lambda} ,E^\sharp[\cdot]]_{|\overline{{\mathbb{B}}_n(0,r)}\setminus\Omega}$ is  continuous from $C^{-1,\alpha}(\overline{\Omega}) $ to  the space $C^{1,\alpha}(\overline{{\mathbb{B}}_n(0,r)}\setminus\Omega)$.
\end{enumerate}  
(See Proposition \ref{prop:nschext} for the definition of $E^\sharp$).
\end{proposition}
Then we also note that the following holds. For a proof, we refer to
\cite[Thm.~18, Prop.~22]{La24d}.
\begin{theorem}\label{thm:fopdvpoho}
 Let $n\in {\mathbb{N}}\setminus\{0,1\}$, $\alpha\in]0,1[$. Let $\Omega $ be a bounded open  subset of ${\mathbb{R}}^n$ of class $C^{1,\alpha}$. Let $\lambda\in {\mathbb{C}}$. Let $S_{n,\lambda}$ be a fundamental solution of the operator $\Delta+\lambda$. If $f\in C^{0,\alpha}(\overline{\Omega})$, then ${\mathcal{P}}_\Omega[S_{n,\lambda},f]$ is continuously differentiable in ${\mathbb{R}}^n$. Moreover, if 
    $l\in\{1,\dots,n\}$, then 
\begin{eqnarray} \label{thm:fopdvpoho1}
\lefteqn{
\frac{\partial}{\partial x_l}{\mathcal{P}}_\Omega[S_{n,\lambda},f](x)
=
{\mathcal{P}}_\Omega[\frac{\partial}{\partial x_l}S_{n,\lambda},f](x)
}
\\ \nonumber
&&\qquad 
={\mathcal{P}}_\Omega\left[S_{n,\lambda},E^\sharp_\Omega\left[ \frac{\partial f}{\partial x_l}\right]\right](x)
-\int_{\partial\Omega}S_{n,\lambda}(x-y)f(y)(\nu_\Omega(y))_l\,d\sigma_y
\end{eqnarray}
for all $x\in {\mathbb{R}}^n$.
\end{theorem}
Then we can immediately deduce the validity for the following statement on   the divergence and curl  of a H\"{o}lder continuous vector volume potential. For the classical version of the formula for the curl, we refer to Kirsch and Hettlich \cite[Lem.~3.26]{KiHe15}.
\begin{corollary}\label{corol:fodivcurlho}
 Let $n\in {\mathbb{N}}\setminus\{0,1\}$. Let $\Omega $ be a bounded open  subset of ${\mathbb{R}}^n$ of class $C^{1,\alpha}$. Let $r\in]0,+\infty[$ be such that
 $\overline{\Omega}\subseteq {\mathbb{B}}_n(0,r)$. Let $\lambda\in {\mathbb{C}}$. Let $S_{n,\lambda}$ be a fundamental solution of the operator $\Delta+\lambda$. If $f\equiv(f_j)_{j=1,\dots,n}\in C^{0,\alpha}(\overline{\Omega}, {\mathbb{C}}^n)$, then $ {\mathcal{P}}_\Omega[S_{n,\lambda},f]\equiv \left( {\mathcal{P}}_\Omega[S_{n,\lambda},f_j]\right)_{j=1,\dots,n}$
 belongs to $C^1({\mathbb{R}}^n, {\mathbb{C}}^n)$, $ {\mathcal{P}}_\Omega^+[S_{n,\lambda},f]$ belongs to $C^{2,\alpha}(\overline{\Omega}, {\mathbb{C}}^n)$ and 
 $ {\mathcal{P}}_\Omega^-[S_{n,\lambda},f]_{\overline{{\mathbb{B}}_n(0,r)}\setminus\Omega}$ belongs to $C^{2,\alpha}(\overline{{\mathbb{B}}_n(0,r)}\setminus\Omega, {\mathbb{C}}^n)$. Moreover,
 \begin{eqnarray}\label{corol:fodivcurlho1}
 {\mathrm{div}}\, \left({\mathcal{P}}_\Omega^\pm[S_{n,\lambda},f]\right)
 ={\mathcal{P}}_\Omega^\pm\left[S_{n,\lambda},E^\sharp_\Omega [{\mathrm{div}}\,f]\right]
 -
 v_\Omega^\pm[S_{n,\lambda}, \nu_\Omega\cdot f_{|\partial\Omega}]\,,
 \end{eqnarray}
 in $\overline{\Omega^\pm}$, where $\Omega^+\equiv\Omega$, $\Omega^-\equiv {\mathbb{R}}^n\setminus\overline{\Omega}$.
 If $n=3$, then we   have
 \begin{equation}\label{corol:fodivcurlho2}
  {\mathrm{curl}}\, \left({\mathcal{P}}_\Omega^\pm[S_{3,\lambda},f]\right)
 ={\mathcal{P}}_\Omega^\pm\left[S_{3,\lambda},E^\sharp_\Omega [{\mathrm{curl}}\,f]\right]
-
 v_\Omega^\pm [S_{3,\lambda}, \nu_\Omega\bar{\wedge} f_{|\partial\Omega}]\,,
\end{equation}
in $\overline{\Omega^\pm}$.
\end{corollary}
{\bf Proof.} The statement follows by Theorem \ref{thm:fopdvpoho}. We just observe the validity of the following equality
\begin{eqnarray*} 
\lefteqn{
 {\mathrm{curl}}\, \left({\mathcal{P}}_\Omega[S_{3,\lambda},f]\right)
}
\\ \nonumber
&&\qquad 
=e_1\left(
\frac{\partial}{\partial x_2}{\mathcal{P}}_\Omega[S_{3,\lambda},f_3]
-
\frac{\partial}{\partial x_3}{\mathcal{P}}_\Omega[S_{3,\lambda},f_2]
\right)
\\ \nonumber
&&\qquad\quad
-e_2\left(
\frac{\partial}{\partial x_1}{\mathcal{P}}_\Omega[S_{3,\lambda},f_3]
-
\frac{\partial}{\partial x_3}{\mathcal{P}}_\Omega[S_{3,\lambda},f_1]
\right)
\\ \nonumber
&&\qquad\quad
+e_3\left(
\frac{\partial}{\partial x_1}{\mathcal{P}}_\Omega[S_{3,\lambda},f_2]
-
\frac{\partial}{\partial x_2}{\mathcal{P}}_\Omega[S_{3,\lambda},f_1]
\right)
\\ \nonumber
&&\qquad 
=e_1\left(
{\mathcal{P}}_\Omega[S_{3,\lambda},E^\sharp_\Omega[
\frac{\partial f_3}{\partial x_2}-\frac{\partial f_2}{\partial x_3}
]]
-v_\Omega [S_{3,\lambda},(\nu_\Omega)_2f_3-(\nu_\Omega)_3f_2]\right)
\\ \nonumber
&&\qquad\quad 
-e_2\left(
{\mathcal{P}}_\Omega[S_{3,\lambda},E^\sharp_\Omega[
\frac{\partial f_3}{\partial x_1}-\frac{\partial f_1}{\partial x_3}
]]
-v_\Omega [S_{3,\lambda},(\nu_\Omega)_1f_3-(\nu_\Omega)_3f_1]\right)
\\ \nonumber
&&\qquad\quad 
+e_3\left(
{\mathcal{P}}_\Omega[S_{3,\lambda},E^\sharp_\Omega[
\frac{\partial f_2}{\partial x_1}-\frac{\partial f_1}{\partial x_2}
]]
-v_\Omega [S_{3,\lambda},(\nu_\Omega)_1f_2-(\nu_\Omega)_2f_1]\right)
\\ \nonumber
&&\qquad
={\mathcal{P}}_\Omega\left[S_{3,\lambda},E^\sharp_\Omega [{\mathrm{curl}}\,f]\right]
-
 v_\Omega[S_{3,\lambda}, \nu_\Omega\bar{\wedge} f_{|\partial\Omega}]\quad\text{in}\	{\mathbb{R}}^3\setminus\partial\Omega\,.
\end{eqnarray*}
and note that by Proposition \ref{prop:dvpsnecr-1a} and by a known continuity result of the single layer potential in ${\mathbb{R}}^3$ (cf.~\textit{e.g.}, \cite[Lem.~4.2 (i), Lem.~6.2]{DoLa17}), 
both the restriction of the right hand side to $\Omega$ and to $\Omega^+$ admit a continuous extension to $\overline{\Omega}$ and to $\overline{\Omega^-}$, respectively. \hfill  $\Box$ 

\vspace{\baselineskip}

   \section{The fundamental theorem of vector calculus   for H\"{o}lder continuous  functions}\label{sec:futvean}
   
We now prove a form of   the fundamental theorem of vector calculus  that 
  generalizes the corresponding classical identity 
(see Martensen \cite[p.~97]{Ma68}, von Wahl \cite[(0.2), p.~125]{vo92}, 
Kirsch and Hettlich \cite[Thm.~3.25]{KiHe15}). We first mention the contribution of von Wahl \cite[p.~125 and following lines]{vo92} in case $\lambda=0$ and the first order partial derivatives belong belong to a Lebesgue space.
As opposed to the classical counterpart, we formulate no assumption whatsoever on the first order partial derivatives of the vector  field by resorting to the operator $E^\sharp_\Omega$ and to the corresponding results for the volume potential of Proposition \ref{prop:dvpsnecr-1a}
\begin{theorem}\label{thm:cudirepl}
Let $\alpha\in]0,1[$. Let $\Omega$ be a bounded open subset of ${\mathbb{R}}^3$ of class $C^{1,\alpha}$. Let $\lambda\in {\mathbb{C}}$.  Let $S_{3,\lambda}$ be a fundamental solution 
of the operator $\Delta+\lambda$. 
Let $E\in C^{0,\alpha}(\overline{\Omega},{\mathbb{C}}^3)$. Then we have
\begin{eqnarray}\nonumber
\lefteqn{
E(x)=-{\mathrm{curl}}_x{\mathcal{P}}_\Omega^+[S_{3,\lambda},E^\sharp_\Omega[{\mathrm{curl}} E]](x)
+D_x{\mathcal{P}}_\Omega^+[S_{3,\lambda},E^\sharp_\Omega[{\mathrm{div}} E]](x)
}
\\ \label{thm:cudirepl1a}
&&\qquad\qquad\qquad\qquad\quad
+\lambda{\mathcal{P}}_\Omega^+[S_{3,\lambda}, E](x)
\\ \nonumber
&&\qquad\qquad\qquad\qquad\quad
+{\mathrm{curl}}_x\int_{\partial\Omega}S_{3,\lambda}(x-y)\nu_\Omega(y)\bar{\wedge}E(y)\,d\sigma_y
\\ \nonumber
&&\qquad\qquad\qquad\qquad\quad
-D_x\int_{\partial\Omega}S_{3,\lambda}(x-y)\nu_\Omega(y)\cdot E(y)\,d\sigma_y
\quad\forall x\in\Omega\,,
\\ \nonumber
\lefteqn{
0=-{\mathrm{curl}}_x{\mathcal{P}}_\Omega^-[S_{3,\lambda},E^\sharp_\Omega[{\mathrm{curl}} E]](x)
+D_x{\mathcal{P}}_\Omega^-[S_{3,\lambda},E^\sharp_\Omega[{\mathrm{div}} E]](x)
}
\\ \label{thm:cudirepl1b}
&&\qquad\qquad\qquad\qquad\quad
+\lambda{\mathcal{P}}_\Omega^-[S_{3,\lambda}, E](x)
\\ \nonumber
&&\qquad\qquad\qquad\qquad\quad
+{\mathrm{curl}}_x\int_{\partial\Omega}S_{3,\lambda}(x-y)\nu_\Omega(y)\bar{\wedge}E(y)\,d\sigma_y
\\ \nonumber
&&\qquad\qquad\qquad\qquad\quad
-D_x\int_{\partial\Omega}S_{3,\lambda}(x-y)\nu_\Omega(y)\cdot E(y)\,d\sigma_y
\quad\forall x\in\Omega^-\,.
\end{eqnarray}
\end{theorem}
{\bf Proof.}  Our starting point are the well-known identities 
\begin{eqnarray*}
E &=&\Delta_x{\mathcal{P}}_\Omega^+[S_{3,\lambda}, E]  
+\lambda{\mathcal{P}}_\Omega^+[S_{3,\lambda}, E] \quad\text{in}\ \Omega\,,
\\ \nonumber
0&=&\Delta_x{\mathcal{P}}_\Omega^-[S_{3,\lambda}, E]  
+\lambda{\mathcal{P}}_\Omega^-[S_{3,\lambda}, E] \quad \text{in}\ \Omega^-\,,
\end{eqnarray*}
that hold in the sense of distributions  for the volume potential (cf. (\ref{prop:nschext3}) and   Proposition \ref{prop:dvpsnecr-1a}). Then the well-known identity
\[
{\mathrm{curl}}{\mathrm{curl}} F=-\Delta F+D{\mathrm{div}} F
\qquad\text{in}\ {\mathbb{R}}^3\setminus\partial\Omega
\]
that holds in the sense of distributions for all $F\in C^1({\mathbb{R}}^3\setminus\partial\Omega)$
and in particular for the volume potential in ${\mathbb{R}}^3\setminus\partial\Omega$
(cf. Proposition \ref{prop:dvpsnecr-1a}))
implies that
\begin{eqnarray}\label{thm:cudirepl2}
\lefteqn{E =-{\mathrm{curl}}\,{\mathrm{curl}}{\mathcal{P}}_\Omega^+[S_{3,\lambda}, E]
}
\\ \nonumber
&&\qquad\qquad
+D{\mathrm{div}}{\mathcal{P}}_\Omega^+[S_{3,\lambda}, E] 
+\lambda{\mathcal{P}}_\Omega^+[S_{3,\lambda}, E] \qquad\text{in}\ \Omega\,,
\\ \nonumber
\lefteqn{0=-{\mathrm{curl}}\,{\mathrm{curl}}{\mathcal{P}}_\Omega^-[S_{3,\lambda}, E]
}
\\ \nonumber
&&\qquad\qquad
+D{\mathrm{div}}{\mathcal{P}}_\Omega^-[S_{3,\lambda}, E] 
+\lambda{\mathcal{P}}_\Omega^-[S_{3,\lambda}, E] \qquad\text{in}\ \Omega^-
\,.
\end{eqnarray}
By Corollary \ref{corol:fodivcurlho}, we have
\begin{eqnarray}\label{thm:cudirepl3}
{\mathrm{curl}}{\mathcal{P}}_\Omega^\pm[S_{3,\lambda}, E]&=&{\mathcal{P}}_\Omega^\pm[S_{3,\lambda}, E^\sharp_\Omega[{\mathrm{curl}} E]]
-v_\Omega^\pm[S_{3,\lambda},\nu_\Omega\bar{\wedge} E]\,,
\\ \nonumber
{\mathrm{div}}{\mathcal{P}}_\Omega^\pm[S_{3,\lambda}, E]&=&{\mathcal{P}}_\Omega^\pm[S_{3,\lambda}, E^\sharp_\Omega[{\mathrm{div}} E]]
-v_\Omega^\pm[S_{3,\lambda},\nu_\Omega\cdot E]
\end{eqnarray}
in $\Omega^\pm$, where $\Omega^+\equiv \Omega$. Let $r\in]0,+\infty[$ be such that
\[
\overline{\Omega}\subseteq {\mathbb{B}}_3(0,r)\,.
\]
 Since $E\in C^{0,\alpha}(\overline{\Omega},{\mathbb{C}}^3)$, we have
\[
{\mathrm{curl}} E\in C^{-1,\alpha}(\overline{\Omega},{\mathbb{C}}^3)\,,
\qquad
{\mathrm{div}} E\in C^{-1,\alpha}(\overline{\Omega},{\mathbb{C}})
\]
and Proposition \ref{prop:dvpsnecr-1a} implies that
\begin{eqnarray*}
&&{\mathcal{P}}_\Omega^+[S_{3,\lambda}, E^\sharp_\Omega[{\mathrm{curl}} E]]\in C^{1,\alpha}(\overline{\Omega},{\mathbb{C}}^3)\,,\ 
{\mathcal{P}}_\Omega^+[S_{3,\lambda}, E^\sharp_\Omega[{\mathrm{div}} E]]  \in C^{1,\alpha}(\overline{\Omega},{\mathbb{C}})\,,
\\ \nonumber
&&{\mathcal{P}}_\Omega^-[S_{3,\lambda}, E^\sharp_\Omega[{\mathrm{curl}} E]]_{|\overline{{\mathbb{B}}_3(0,r)}\setminus\Omega}\in C^{1,\alpha}(\overline{{\mathbb{B}}_3(0,r)}\setminus\Omega,{\mathbb{C}}^3)\,,
\\ \nonumber
&&
{\mathcal{P}}_\Omega^-[S_{3,\lambda}, E^\sharp_\Omega[{\mathrm{div}} E]]_{|\overline{{\mathbb{B}}_3(0,r)}\setminus\Omega}\in C^{1,\alpha}(\overline{{\mathbb{B}}_3(0,r)}\setminus\Omega,{\mathbb{C}})\,.
\end{eqnarray*}
Since
\[
\nu_\Omega\bar{\wedge} E\in C^{0,\alpha}(\partial\Omega,{\mathbb{C}}^3)
\,\qquad
\nu_\Omega\cdot E\in C^{0,\alpha}(\partial\Omega,{\mathbb{C}})\,,
\]
a classical result on the single layer potential   implies that
\begin{eqnarray*}
&&
v_\Omega^+[S_{3,\lambda},\nu_\Omega\bar{\wedge} E]\in C^{1,\alpha}(\overline{\Omega},{\mathbb{C}}^3)\,,
 \qquad
v_\Omega^+[S_{3,\lambda},\nu_\Omega\cdot E]\in C^{1,\alpha}(\overline{\Omega},{\mathbb{C}}) 
\\ \nonumber
&&v_\Omega^-[S_{3,\lambda},\nu_\Omega\bar{\wedge} E]_{|\overline{{\mathbb{B}}_3(0,r)}\setminus\Omega}\in C^{1,\alpha}(\overline{{\mathbb{B}}_3(0,r)}\setminus\Omega,{\mathbb{C}}^3)
\,,\\ \nonumber
&&
v_\Omega^-[S_{3,\lambda},\nu_\Omega\cdot E]_{|\overline{{\mathbb{B}}_3(0,r)}\setminus\Omega}\in C^{1,\alpha}(\overline{{\mathbb{B}}_3(0,r)}\setminus\Omega,{\mathbb{C}}) 
\end{eqnarray*}
(see Miranda \cite{Mi65}, \cite[Thm.~7.1]{DoLa17}).
Hence, the equalities (\ref{thm:cudirepl2}) and (\ref{thm:cudirepl3})  imply that
\begin{eqnarray*} 
\lefteqn{
E=-{\mathrm{curl}}\,
{\mathcal{P}}_\Omega^+[S_{3,\lambda}, E^\sharp_\Omega[{\mathrm{curl}} E]] 
+{\mathrm{curl}}\,v_\Omega^+[S_{3,\lambda},\nu_\Omega\bar{\wedge} E] 
}
\\ \nonumber
&&\quad
+D{\mathcal{P}}_\Omega^+[S_{3,\lambda}, E^\sharp_\Omega[{\mathrm{div}} E]]
-Dv_\Omega^+[S_{3,\lambda},\nu_\Omega\cdot E]+\lambda{\mathcal{P}}_\Omega^+[S_{3,\lambda}, E]\quad\text{in}\ \Omega\,,
\\ \nonumber
\lefteqn{
0=-{\mathrm{curl}}\,
{\mathcal{P}}_\Omega^-[S_{3,\lambda}, E^\sharp_\Omega[{\mathrm{curl}} E]] 
+{\mathrm{curl}}\,v_\Omega^-[S_{3,\lambda},\nu_\Omega\bar{\wedge} E] 
}
\\ \nonumber
&&\quad
+D{\mathcal{P}}_\Omega^-[S_{3,\lambda}, E^\sharp_\Omega[{\mathrm{div}} E]]
-Dv_\Omega^-[S_{3,\lambda},\nu_\Omega\cdot E]+\lambda{\mathcal{P}}_\Omega^-[S_{3,\lambda}, E]\quad\text{in}\ \Omega^-\,,
\end{eqnarray*}
and thus the proof of formulas (\ref{thm:cudirepl1a}), (\ref{thm:cudirepl1b}) is complete.\hfill  $\Box$ 

\vspace{\baselineskip}

\section{A preliminary statement  on the Neumann problem}\label{sec:prelNeum}

We first introduce the following preliminary statement that follows by   \cite[Thm.~7.23]{DaLaMu21} 
{\color{black}for	}  $m=1$ and by classical results in potential theory for $m\geq 2$. (See Section \ref{sec:prelnot} for the notation on the connected components $\Omega_j$  of $\Omega$ with $j\in \{1,\dots,\varkappa^+\}$).
  \begin{proposition}\label{prop:gintneuposchma}
 Let   $\alpha\in ]0,1[$, $m\in{\mathbb{N}}\setminus\{0\}$. Let $\Omega$ be a bounded open  subset of ${\mathbb{R}}^{n}$ of class $C^{m,\alpha}$. Let
 \begin{eqnarray}\nonumber
\lefteqn{X^{m,\alpha}\equiv\left\{u\in C^{m,\alpha}(\overline{\Omega}) :\,
\int_{\Omega_j}u\,dx=0\ \forall j\in \{1,\dots,\varkappa^+\}
\right\}\,,}
\\ \nonumber
\lefteqn{Y_{m,\alpha}\equiv\biggl\{
 (f,g)\in C^{m-2,\alpha}( \overline{\Omega})\times C^{m-1,\alpha}(\partial\Omega):\,
 }
 \\ \label{prop:gintneuposchma0}
 &&\qquad
\int_{\partial\Omega_j}g\,d\sigma= \langle E^\sharp_{\Omega_j}[f_{|\Omega_j}],1\rangle \ \forall j\in \{1,\dots,\varkappa^+\}
 \biggr\}\,,
\end{eqnarray}
where $\langle E^\sharp_{\Omega_j}[f_{|\Omega_j}],1\rangle=\int_{\Omega_j}f\,dx$ for all $j\in \{1,\dots,\varkappa^+\}$ in case $m\geq 2$ (cf. (\ref{prop:nschext3})). 
Then there exists
$ 
G_{m,\alpha}\in {\mathcal{L}}\left(
Y_{m,\alpha}
,C^{m,\alpha}(\overline{\Omega}) 
\right)
$
such that 
\begin{equation}\label{prop:gintneuposchma1}
(\Delta, \partial_{\nu_\Omega} )[G_{m,\alpha}(f,g)]=(f,g)\qquad\forall (f,g)\in Y_{m,\alpha}\,.
\end{equation}
\end{proposition}
{\bf Proof.} By the classical uniqueness theorem up to locally constant functions for the Neumann problem and by the known compatibility conditions for the data of the Neumann problem,  the restriction of $(\Delta, \partial_{\nu_\Omega} )$ to the closed subspace $X^{m,\alpha}$ 
of the Banach space $C^{m,\alpha}(\overline{\Omega})$ is injective and maps $X^{m,\alpha}$  to the Banach space 
$Y_{m,\alpha}$ (cf.~\textit{e.g.}, \cite[Thms.~4.1, 6.2, 7.23]{DaLaMu21}). Thus it suffices to show that the restriction of  $(\Delta, \partial_{\nu_\Omega} )$ to $X^{m,\alpha}$ is surjective onto $Y_{m,\alpha}$. Indeed, the Open Mapping Theorem implies that the inverse $G_{m,\alpha}$ of such a restriction satisfies the conditions of the statement.

For a proof of the surjectivity of $(\Delta, \partial_{\nu_\Omega} )$ in case  $m=1$, we refer to \cite[Thm.~7.23]{DaLaMu21}.  If $m\geq 2$ and $(f,g)\in Y_{m,\alpha}$, then (\ref{prop:nschext3}) and Theorem \ref{thm:nwtdma} (i) on the volume potential imply that
\[
{\mathcal{P}}_\Omega^+[S_n,E^\sharp_\Omega[f ]]={\mathcal{P}}_\Omega^+[S_n,f ]\in C^{m,\alpha}(\overline{\Omega})\,.
\]
Thus it suffices to show that  the Neumann problem
\begin{equation}
 \label{prop:gintneuposchma2}
\left\{
\begin{array}{ll}
\Delta h=0 &{\mathrm{in}}\ \Omega\,,
\\
 \partial_{\nu_\Omega} h =g- \frac{\partial}{\partial_{\nu_\Omega}}{\mathcal{P}}_\Omega^+[{\color{black}S_n,}f]    &{\mathrm{on}}\ \partial\Omega 
\end{array}
\right.
\end{equation}
has a solution $h\in C^{m,\alpha}(\overline{\Omega})$. Indeed, possibly adding a function which is constant on each connected components of $\Omega$ to the function $h+{\mathcal{P}}_\Omega^+[S_n,f]$, we obtain a solution $u\in X^{m,\alpha}$ of equation (\ref{prop:gintneuposchma1}). By the Divergence Theorem in $\Omega_j$ and by the membership of $(f,g)$ in $ Y_{m,\alpha}$, we have
\begin{equation}
 \label{prop:gintneuposchma3}
\int_{\partial\Omega_j}g- \frac{\partial}{\partial \nu_\Omega}{\mathcal{P}}_\Omega^+[{\color{black}S_n,}f] \,d\sigma=0
\end{equation}
for all $j\in \{1,\dots,\varkappa^+\}$. Now let  $W_\Omega[S_n,\cdot]$ be the boundary integral operator that corresponds to the double layer potential as in (\ref{thm:dlaybdry}) with $\lambda=0$. It is known that 
$W_\Omega[S_n,\cdot]$ is compact in $C^{m,\alpha}(\partial\Omega)$ and that its transpose operator
$W_\Omega^t[S_n,\cdot]$ is compact from $C^{m-1,\alpha}(\partial\Omega)$ to itself (cf.~\textit{e.g.}, Dondi and the author \cite[Cors.~9.1, 10.1]{DoLa17}).  Then  the Fredholm Alternative Theorem in the duality pairing   
\begin{equation}\label{prop:KerI+W21}
\langle C^{m-1,\alpha}(\partial\Omega), C^{m,\alpha}(\partial\Omega)\rangle
\end{equation}
in the form of Wendland \cite{We67}, \cite{We70} (cf. Kress~\cite[Thm.~4.17]{Kr14}) together with the coincidence of the null space
\[
\left\{\mu\in C^{m,\alpha}(\partial\Omega):\,\left(-\frac{1}{2}I+W_{\Omega}\right)[\mu]=0\right\}
\]
with the set of the functions from $\partial\Omega$ to $\mathbb{R}$ which are constant on $\partial\Omega_j$ for all $j\in\{1,\dots,\varkappa^+\}$  (cf.~\textit{e.g.}, \cite[Thm.~A.18]{La24b}) implies that the integral equation
\[
-\frac{1}{2}\eta +W_\Omega^t[S_n,\eta]=g- \frac{\partial}{\partial \nu_\Omega}{\mathcal{P}}_\Omega^+[f]  
\]
has a solution $\eta\in C^{m-1,\alpha}(\partial\Omega)$. Then $h\equiv v_\Omega^+[S_n,\eta]$ belongs to $C^{m,\alpha}(\overline\Omega)$ and the classical jump formula  for the normal derivative   of the  single layer potential implies that $h$ solves the above problem (\ref{prop:gintneuposchma2}) 
and thus the proof is complete
(cf.~\textit{e.g.}, \cite[Thm.~7.1]{DoLa17}).\hfill  $\Box$ 

\vspace{\baselineskip}

\section{The Helmholtz-Weil decomposition in H\"older spaces}\label{sec:helweyl}

 We wish to prove the Helmholtz-Weyl decomposition   in H\"older spaces
by writing an explicit form of the projection operators   that develops from an idea of von Wahl \cite{vo90} in the context of Lebesgue spaces. Here we have to modify the arguments of von Wahl \cite{vo90} also because the functions of class $C^\infty$ with compact support are not dense in a H\"{o}lder space. In particular, we exploit a formula for the projectors that is a modification of that of Wahl \cite{vo90}. In order to prove it, we resort to 
the fundamental theorem of vector calculus
 for H\"older continuous vector fields in the form of Theorem \ref{thm:cudirepl}. 
We first introduce the following known elementary statement. For the convenience of the reader, we include a proof. 
\begin{lemma}\label{lem:predeco} Let $\Omega$ be a bounded open Lipschitz subset of ${\mathbb{R}}^n$. Let $\alpha\in ]0,1]$. Let
\begin{eqnarray} \nonumber
\lefteqn{
C^{0,\alpha}_{{\mathrm{div}}=0,t}(\overline{\Omega},{\mathbb{C}}^3)\equiv\biggl\{
E\in C^{0,\alpha}(\overline{\Omega},{\mathbb{C}}^3):\,
}\\ \nonumber
&&\qquad\qquad\qquad\qquad 
\nu_\Omega\cdot E_{|\partial\Omega}=0\ \text{on}\ \partial\Omega\,, \ 
{\mathrm{div}}\,E=0\ \text{in}\ \Omega
\biggr\}\,,
\\ \label{lem:predeco1}
\lefteqn{
\Xi_2=C^{0,\alpha}_{\nabla}(\overline{\Omega},{\mathbb{C}}^3)\equiv \biggl\{\nabla \phi:\, \phi\in C^{1,\alpha}(\overline{\Omega},{\mathbb{C}})\biggr\}\,.
}
\end{eqnarray}
Then the following statements hold.
\begin{enumerate}
\item[(i)] $C^{0,\alpha}_{{\mathrm{div}}=0,t}(\overline{\Omega},{\mathbb{C}}^3)$ and $C^{0,\alpha}_{\nabla}(\overline{\Omega},{\mathbb{C}}^3)$ are closed subspaces of $C^{0,\alpha}(\overline{\Omega},{\mathbb{C}}^3)$. 
\item[(ii)]  If $E$ belongs to $C^{0,\alpha}_{{\mathrm{div}}=0,t}(\overline{\Omega},{\mathbb{C}}^3)$ and to $C^{0,\alpha}_{\nabla}(\overline{\Omega},{\mathbb{C}}^3)$, then the conjugate vector field $\overline{E}$ belongs to $C^{0,\alpha}_{{\mathrm{div}}=0,t}(\overline{\Omega},{\mathbb{C}}^3)$ and to $C^{0,\alpha}_{\nabla}(\overline{\Omega},{\mathbb{C}}^3)$, respectively.
\item[(iii)] If $E\in C^{0,\alpha}_{{\mathrm{div}}=0,t}(\overline{\Omega},{\mathbb{C}}^3)$ and $\phi\in C^{1,\alpha}(\overline{\Omega},{\mathbb{C}})$, then 
$\int_\Omega E\cdot \nabla\phi\,dx=0$. 
\item[(iv)] $C^{0,\alpha}_{{\mathrm{div}}=0,t}(\overline{\Omega},{\mathbb{C}}^3)\cap C^{0,\alpha}_{\nabla}(\overline{\Omega},{\mathbb{C}}^3)=\{0\}$.
\end{enumerate}
\end{lemma}
{\bf Proof.} (i) Since the components of $\nu_\Omega$ belong to  $C^{0,\alpha}(\partial\Omega)$ and the pointwise product is continuous in the space $C^{0,\alpha}(\partial\Omega)$, the map from
$C^{0,\alpha}(\overline{\Omega},{\mathbb{C}}^3)$ to $C^{0,\alpha}(\partial\Omega)$ that takes $E$ to $\nu_\Omega\cdot E_{|\partial\Omega}$ is continuous. Since the divergence is continuous from $C^{0,\alpha}(\overline{\Omega},{\mathbb{C}}^3)$ to $C^{-1,\alpha}(\overline{\Omega},{\mathbb{C}}^3)$, we conclude that
$C^{0,\alpha}_{{\mathrm{div}}=0,t}(\overline{\Omega},{\mathbb{C}}^3)$ is a closed subspace of $C^{0,\alpha}(\overline{\Omega},{\mathbb{C}}^3)$. Next, we note that
\begin{equation} \label{lem:predeco2}
C^{0,\alpha}_{\nabla}(\overline{\Omega},{\mathbb{C}}^3)=\bigcap_{\gamma\ \text{is\ a}\ C^1\ \text{closed\ path\ in}\ \Omega }\left\{g\in C^{0,\alpha}(\overline{\Omega},{\mathbb{C}}^3):\,
\int_\gamma g\,dl=0\right\}\,,
\end{equation}
where the integrals in the right hand side are line integrals. Indeed, the space $C^{0,\alpha}_{\nabla}(\overline{\Omega},{\mathbb{C}}^3)$ is well known to be contained in the right hand side and if $g$ belongs to the right hand side, then there exists $\phi\in C^1(\Omega)$ such that $g=\nabla\phi$. Since $g$ is bounded, we conclude that $\phi$ is Lipschitz continuous and accordingly $\phi$ has a continuous extension to $\overline{\Omega}$. Since $g=\nabla\phi$, we conclude that $\phi\in C^{1,\alpha}(\overline{\Omega})$, \textit{i.e.}, $g$ belongs to $C^{0,\alpha}_{\nabla}(\overline{\Omega},{\mathbb{C}}^3)$ and thus the proof of equality  (\ref{lem:predeco2}) is complete. Since
the norm of $C^{0,\alpha}(\overline{\Omega},{\mathbb{C}}^3)$ is stronger than the uniform convergence, 
each set in the right hand side of (\ref{lem:predeco2}) is closed in $C^{0,\alpha}(\overline{\Omega},{\mathbb{C}}^3)$ and accordingly the intersection in the right hand side of (\ref{lem:predeco2}) is closed in $C^{0,\alpha}(\overline{\Omega},{\mathbb{C}}^3)$. The validity of (ii) is an immediate consequence of the definition of $C^{0,\alpha}_{{\mathrm{div}}=0,t}(\overline{\Omega},{\mathbb{C}}^3)$ and of $C^{0,\alpha}_{\nabla}(\overline{\Omega},{\mathbb{C}}^3)$, respectively.

(iii) By Proposition \ref{prop:gendidis} and by our assumptions on $E$ and $\phi$,  we have
\[
\int_\Omega E\cdot \nabla\phi\,dx
=\int_{\partial\Omega}(\nu_\Omega\cdot E)  \phi\,d\sigma\
-\langle E^\sharp_\Omega[{\mathrm{div}}\, E],\phi\rangle =0\,.
\]
(iv) Let  $E\in C^{0,\alpha}_{{\mathrm{div}}=0,t}(\overline{\Omega},{\mathbb{C}}^3)\cap C^{0,\alpha}_{\nabla}(\overline{\Omega},{\mathbb{C}}^3)$. By statement (ii),   $\overline{E}$ belongs to $ C^{0,\alpha}_{\nabla}(\overline{\Omega},{\mathbb{C}}^3)$ and accordingly, statement (iii) implies that $\int_\Omega E\cdot \overline{E}\,dx=0$. Hence, $E=0$ and statement (iv) holds true.\hfill  $\Box$ 

\vspace{\baselineskip}

We are now ready to prove the following statement that extends to H\"{o}lder spaces Theorem 4.1 of
von Wahl \cite[p.~26]{vo90} in the context of Lebesgue spaces. We also note that the formulas
(\ref{thm:hewedeco0a2})--(\ref{thm:hewedeco0a4}) below are  modifications of corresponding formulas of von Wahl \cite[p.~27]{vo90}.
For the basic definitions of vector sums and projections, we refer, \textit{e.g.}, to Rudin~\cite[\S 4.20, \S 5.15]{Ru91}

\begin{theorem}\label{thm:hewedeco0a}
 Let $\Omega $ be a bounded open  subset of ${\mathbb{R}}^3$ of class $C^{1,\alpha}$. Then 
\begin{equation}\label{thm:hewedeco0a1}
 C^{0,\alpha}(\overline{\Omega},{\mathbb{C}}^3)=C^{0,\alpha}_{{\mathrm{div}}=0,t}(\overline{\Omega},{\mathbb{C}}^3)\oplus C^{0,\alpha}_{\nabla}(\overline{\Omega},{\mathbb{C}}^3)\,,
\end{equation}
where the direct sum is topological. Let  $\Pi_1$ be the map from  $C^{0,\alpha}(\overline{\Omega},{\mathbb{C}}^3)$ to $C^{0,\alpha}_{{\mathrm{div}}=0,t}(\overline{\Omega},{\mathbb{C}}^3)$ that is delivered by the formula
\begin{equation}\label{thm:hewedeco0a2}
\Pi_1[E]\equiv 
-{\mathrm{curl}} \biggl\{
{\mathcal{P}}_\Omega^+[S_3,E^\sharp_\Omega[{\mathrm{curl}} E]] 
-v^+_\Omega[S_3,\nu_\Omega \bar{\wedge}E] \biggr\}+\nabla H[E]
\end{equation}
where
 $H[E]$ is the only solution in $C^{1,\alpha}(\overline{\Omega})$ of the Neumann problem
 \begin{equation}\label{thm:hewedeco0a3}
\left\{
\begin{array}{ll}
 \Delta H[E]=0 &\text{in}\ \Omega\,,
 \\
 \frac{\partial}{\partial\nu_\Omega}H[E] &
 \\\quad
 =\nu_\Omega\cdot {\mathrm{curl}} \biggl\{
{\mathcal{P}}_\Omega^+[S_{3},E^\sharp_\Omega[{\mathrm{curl}} E]] 
-v^+_\Omega[S_3,\nu_\Omega \bar{\wedge}E]\biggr\} 
  &\text{on}\ \partial\Omega\,,
  \\
  \int_{\Omega_j}H[E]\,dx=0 & \forall j\in\{1,\dots,\varkappa^+\}\,,
\end{array}
\right.
\end{equation}
for all $E\in C^{0,\alpha}(\overline{\Omega},{\mathbb{C}}^3)$ (For the definition of $\varkappa^+$, see right after (\ref{eq:exto})). Then $\Pi_1$
is a linear and continuous projection of $C^{0,\alpha}(\overline{\Omega},{\mathbb{C}}^3)$ onto its subspace  $C^{0,\alpha}_{{\mathrm{div}}=0,t}(\overline{\Omega},{\mathbb{C}}^3)$ along $C^{0,\alpha}_{\nabla}(\overline{\Omega},{\mathbb{C}}^3)$, \textit{i.e.}, 
\[
C^{0,\alpha}_{\nabla}(\overline{\Omega},{\mathbb{C}}^3)={\mathrm{Ker}}\, \Pi_1
\]
and 
the map $\Pi_2$ from $C^{0,\alpha}(\overline{\Omega},{\mathbb{C}}^3)$ to $C^{0,\alpha}_{\nabla}(\overline{\Omega},{\mathbb{C}}^3)$ that is delivered by the formula
\begin{equation}\label{thm:hewedeco0a4}
\Pi_2[E]\equiv \nabla
\biggl\{
{\mathcal{P}}_\Omega^+[S_{3},E^\sharp_\Omega[{\mathrm{div}} E]]
-v^+_\Omega[S_{3},\nu_\Omega\cdot E]
\biggr\}-\nabla H[E]
\end{equation}
 for all $E\in C^{0,\alpha}(\overline{\Omega},{\mathbb{C}}^3)$
is a linear and continuous projection of $C^{0,\alpha}(\overline{\Omega},{\mathbb{C}}^3)$ onto its subspace  $C^{0,\alpha}_{\nabla}(\overline{\Omega},{\mathbb{C}}^3)$ along $C^{0,\alpha}_{{\mathrm{div}}=0,t}(\overline{\Omega},{\mathbb{C}}^3)$ , \textit{i.e.}, 
\[
C^{0,\alpha}_{{\mathrm{div}}=0,t}(\overline{\Omega},{\mathbb{C}}^3)={\mathrm{Ker}}\, \Pi_2 
\,.
\]
 \end{theorem}
{\bf Proof.} We first consider the Neumann problem (\ref{thm:hewedeco0a3}). By Proposition 
\ref{prop:dvpsnecr-1a}, the map ${\mathcal{P}}_\Omega^+[S_{3},E^\sharp_\Omega[{\mathrm{curl}} [\cdot]]]$  is linear and continuous from $C^{0,\alpha}(\overline{\Omega},{\mathbb{C}}^3)$ to
$C^{1,\alpha}(\overline{\Omega},{\mathbb{C}}^3)$.
By the continuity of the pointwise product in $C^{0,\alpha}(\partial\Omega)$ and by classical regularity properties of the acoustic single layer potential, the map from $C^{0,\alpha}(\overline{\Omega},{\mathbb{C}}^3)$ to $C^{1,\alpha}(\overline{\Omega},{\mathbb{C}}^3)$ that takes $E$ to $v_\Omega^+ [S_3,\nu_\Omega\bar{\wedge}E ]$ is linear and continuous (cf.~\textit{e.g.}, \cite[Thm.~7.1]{DoLa17}). Hence, the map $G$ from
$C^{0,\alpha}(\overline{\Omega},{\mathbb{C}}^3)$ to $C^{0,\alpha}(\partial \Omega,{\mathbb{C}}^3)$
that takes $E$ to
\[
G[E]\equiv \biggl\{\nu_\Omega\cdot{\mathrm{curl}} \biggl\{
{\mathcal{P}}_\Omega^+[S_3,E^\sharp_\Omega[{\mathrm{curl}} E]] 
-v^+_\Omega[S_3,\nu_\Omega \bar{\wedge}E]\biggr\} \biggr\}_{|\partial\Omega}
\]
is linear and continuous. By Lemma \ref{lem:nucuv} {\color{black}and } the Definition \ref{defn:tandiv} of tangential divergence, we have
\begin{eqnarray*} 
\lefteqn{
\int_{\partial\Omega_j}G[E]\cdot\nu_\Omega\,d\sigma
=\int_{\partial\Omega_j} \nu_\Omega\cdot {\mathrm{curl}} \biggl\{
{\mathcal{P}}_\Omega^+[S_3,E^\sharp_\Omega[{\mathrm{curl}} E]] 
-v^+_\Omega[S_3,\nu_\Omega \bar{\wedge}E]\biggr\} \,d\sigma
}
\\ \nonumber
&& 
=-\left\langle{\mathrm{div}}_{\partial\Omega}\biggl(
\nu_\Omega\bar{\wedge}\biggl\{
{\mathcal{P}}_\Omega^+[S_3,E^\sharp_\Omega[{\mathrm{curl}} E]] 
-v^+_\Omega[S_3,\nu_\Omega \bar{\wedge}E]\biggr\}
\biggr),\chi_{\partial\Omega_j}\right\rangle=0\,,
\end{eqnarray*}
for all $j\in\{1,\dots,\varkappa^+\}$
and accordingly $(0, G[E])\in Y_{1,\alpha}$ for each 
$E$ {\color{black}in	}	$C^{0,\alpha}(\overline{\Omega},{\mathbb{C}}^3)$ (see (\ref{prop:gintneuposchma0}) for the definition of $Y_{1,\alpha}$).  Thus  Proposition \ref{prop:gintneuposchma} implies that the map $H$ from $C^{0,\alpha}(\overline{\Omega},{\mathbb{C}}^3)$ to $C^{1,\alpha}(\overline{\Omega})$ that is delivered by 
\[
H[E]\equiv G_{1,\alpha}(0, G[E])\qquad \forall E\in C^{0,\alpha}(\overline{\Omega},{\mathbb{C}}^3)
\]
is linear and continuous and $H[E]$ solves the Neumann problem (\ref{thm:hewedeco0a3}) for all 
$E\in C^{0,\alpha}(\overline{\Omega},{\mathbb{C}}^3)$. Since  $H[E]$ satisfies the Neumann boundary condition of problem (\ref{thm:hewedeco0a3}), we have 
\[
\nu_\Omega\cdot \Pi_1[E]=0\qquad\forall E\in C^{0,\alpha}(\overline{\Omega},{\mathbb{C}}^3)\,.
\]
Since $H[E]$ is harmonic in $\Omega$, the known identities
\[
{\mathrm{div}}\,{\mathrm{curl}}=0\,,\qquad {\mathrm{div}}\,\nabla=\Delta
\]
imply that
\[
{\mathrm{div}}\,\Pi_1[E]=0\qquad\forall E\in C^{0,\alpha}(\overline{\Omega},{\mathbb{C}}^3)\,.
\]
Hence, 
\[
\Pi_1[E]\in C^{0,\alpha}_{{\mathrm{div}}=0,t}(\overline{\Omega},{\mathbb{C}}^3)\qquad\forall E\in C^{0,\alpha}(\overline{\Omega},{\mathbb{C}}^3)\,.
\] 
Moreover, the above mentioned classical regularity results and the above mentioned linearity and continuity  of $H[\cdot]$ imply that $\Pi_1$ is linear and continuous from $C^{0,\alpha}(\overline{\Omega},{\mathbb{C}}^3)$ to $C^{0,\alpha}_{{\mathrm{div}}=0,t}(\overline{\Omega},{\mathbb{C}}^3)$.\par

By the fundamental theorem of vector calculus in the form of Theorem \ref{thm:cudirepl} with $\lambda=0$, we have 
\begin{equation}\label{thm:hewedeco0a5}
E=\Pi_1[E]+\Pi_2[E]\qquad\forall E\in C^{0,\alpha}(\overline{\Omega},{\mathbb{C}}^3)\,.
\end{equation}
Thus if  $E\in C^{0,\alpha}_{{\mathrm{div}}=0,t}(\overline{\Omega},{\mathbb{C}}^3)$, then equality (\ref{thm:hewedeco0a5}) implies that
\[
\Pi_2[E]= E-\Pi_1[E]\in C^{0,\alpha}_{{\mathrm{div}}=0,t}(\overline{\Omega},{\mathbb{C}}^3)-C^{0,\alpha}_{{\mathrm{div}}=0,t}(\overline{\Omega},{\mathbb{C}}^3)=C^{0,\alpha}_{{\mathrm{div}}=0,t}(\overline{\Omega},{\mathbb{C}}^3)\,.
\]
 On the other hand, 
the definition  (\ref{thm:hewedeco0a4}) of $\Pi_2$ and the membership of $E$ in $C^{0,\alpha}_{{\mathrm{div}}=0,t}(\overline{\Omega},{\mathbb{C}}^3)$ imply that
\[
\Pi_2[E]=-\nabla H[E]\in C^{0,\alpha}_{\nabla}(\overline{\Omega},{\mathbb{C}}^3)\,.
\]
 Hence, $\Pi_2[E]\in C^{0,\alpha}_{{\mathrm{div}}=0,t}(\overline{\Omega},{\mathbb{C}}^3)\cap C^{0,\alpha}_{\nabla}(\overline{\Omega},{\mathbb{C}}^3)$ and accordingly 
Lemma \ref{lem:predeco} (iv) implies that $\Pi_2[E]=0$. Then again equality (\ref{thm:hewedeco0a5}) implies that 
$\Pi_1[E]=E$. Hence, $\Pi_1$ is a projection onto $C^{0,\alpha}_{{\mathrm{div}}=0,t}(\overline{\Omega},{\mathbb{C}}^3)$. By equality (\ref{thm:hewedeco0a5}), we also have $\Pi_2=I-\Pi_1$ and accordingly, $\Pi_2$ is   a projection onto 
 ${\mathrm{Im}}\, (I-\Pi_1)$ and 
\begin{eqnarray*} 
\lefteqn{
C^{0,\alpha}(\overline{\Omega},{\mathbb{C}}^3)={\mathrm{Im}}\,\Pi_1\oplus {\mathrm{Im}}\, (I-\Pi_1)
}
\\ \nonumber
&&\qquad\qquad\qquad 
 ={\mathrm{Im}}\,\Pi_1\oplus {\mathrm{Im}}\,\Pi_2=C^{0,\alpha}_{{\mathrm{div}}=0,t}(\overline{\Omega},{\mathbb{C}}^3)\oplus {\mathrm{Im}}\,\Pi_2\,,
\end{eqnarray*}
 where the direct sums are topological. 
By the linearity and continuity of $H[\cdot]$ and by the same classical results that we have invoked to prove that $\Pi_1$ is linear and continuous,  
the operator from   $C^{0,\alpha}(\overline{\Omega},{\mathbb{C}}^3)$ to $C^{1,\alpha}(\overline{\Omega})$ that takes $E$ to 
\[
{\mathcal{P}}_\Omega^+[S_3,E^\sharp_\Omega[{\mathrm{div}} E]]
-v^+_\Omega[S_3,\nu_\Omega\cdot E]-H[E]
\]
 is linear and continuous. Then 
  the definition of $\Pi_2$, implies that $\Pi_2$ is linear and continuous from $C^{0,\alpha}(\overline{\Omega},{\mathbb{C}}^3)$ to the subspace $C^{0,\alpha}_{\nabla}(\overline{\Omega},{\mathbb{C}}^3)$ of $C^{0,\alpha}(\overline{\Omega},{\mathbb{C}}^3)$. We now prove that  the image of $\Pi_2$ equals $C^{0,\alpha}_{\nabla}(\overline{\Omega},{\mathbb{C}}^3)$. Assume by contradiction that there exists $F\in C^{0,\alpha}_{\nabla}(\overline{\Omega},{\mathbb{C}}^3)\setminus {\mathrm{Im}}\,\Pi_2$. Then the membership of 
  $F$ in $C^{0,\alpha}_{\nabla}(\overline{\Omega},{\mathbb{C}}^3)$, the membership of $\Pi_1[F]$ in $C^{0,\alpha}_{{\mathrm{div}}=0,t}(\overline{\Omega},{\mathbb{C}}^3)$,
  the membership of $\Pi_2[F]$ in $C^{0,\alpha}_{\nabla}(\overline{\Omega},{\mathbb{C}}^3)$,
   and Lemma \ref{lem:predeco} (ii), (iii) imply that
\begin{eqnarray*} 
\lefteqn{
0=\int_\Omega  \Pi_1[F]\cdot \overline{F}\,dx=\int_\Omega  \Pi_1[F]\cdot\overline{\Pi_1[F]+\Pi_2[F]}\,dx
}
\\ \nonumber
&&\qquad 
=\int_\Omega  |\Pi_1[F]|^2\,dx+\int_\Omega  \Pi_1[F]\overline{\Pi_2[F]}\,dx=\int_\Omega  |\Pi_1[F]|^2\,dx\,.
\end{eqnarray*}
  Hence, $\Pi_1[F]=0$ and thus $F=\Pi_1[F]+\Pi_2[F]=\Pi_2[F]\in  {\mathrm{Im}}\,\Pi_2$, a contradiction. 
Accordingly, $  {\mathrm{Im}}\,\Pi_2=C^{0,\alpha}_{\nabla}(\overline{\Omega},{\mathbb{C}}^3)$. Since $\Pi_2$ is   a projection onto 
 $ {\mathrm{Im}}\,\Pi_2={\mathrm{Ker}}\, (I-\Pi_2)$ (cf.~\textit{e.g.}, Rudin \cite[5.15 (a)]{Ru91}), we have
 \[
 C^{0,\alpha}_{\nabla}(\overline{\Omega},{\mathbb{C}}^3)= {\mathrm{Im}}\,\Pi_2={\mathrm{Ker}}\, (I-\Pi_2)
 ={\mathrm{Ker}}\,\Pi_1\,.
 \]
 Since $\Pi_1$ is   a projection onto 
 ${\mathrm{Im}}\, \Pi_1={\mathrm{Ker}}\, (I-\Pi_1)$ (cf.~\textit{e.g.}, Rudin \cite[5.15 (a)]{Ru91}), we have 
 \[
 C^{0,\alpha}_{{\mathrm{div}}=0,t}(\overline{\Omega},{\mathbb{C}}^3)={\mathrm{Im}}\,\Pi_1={\mathrm{Ker}}\, (I-\Pi_1)
 ={\mathrm{Ker}}\,\Pi_2
 \]
and thus the proof is complete.\hfill  $\Box$ 

\vspace{\baselineskip}

 \noindent
{\bf Acknowledgement}
The author  acknowledges  the support of the Gruppo Nazionale per l'Analisi Matematica, la Probabilit\`a e le loro Applicazioni (GNAMPA) of the Istituto Nazionale di Alta Matematica (INdAM).

\end{document}